\documentclass[letterpaper]{article} 
\usepackage{amsmath}
\usepackage{amsmath,amsfonts,amsthm,amssymb}
\usepackage{aaai25}  
\usepackage{times}  
\usepackage{helvet}  
\usepackage{courier}  
\usepackage[hyphens]{url}  
\usepackage{graphicx} 
\usepackage{natbib}  
\usepackage{caption} 
\usepackage{algorithm}
\usepackage{algorithmic}

\usepackage{newfloat}
\usepackage{listings}
\DeclareCaptionStyle{ruled}{labelfont=normalfont,labelsep=colon,strut=off} 
\floatstyle{ruled}
\newfloat{listing}{tb}{lst}{}
\floatname{listing}{Listing}
\newtheorem{thm}{Theorem}

\newtheorem{lemma}{Lemma}

\newtheorem{corollary}{Corollary}
\newtheorem{assumption}{Assumption}

\DeclareMathOperator*{\argmax}{argmax}
\DeclareMathOperator*{\argmin}{argmin}

\title{Convergence Rate in Nonlinear Two-Time-Scale Stochastic Approximation 
\\with State (Time)-Dependence}
\author{
    Zixi Chen\textsuperscript{\rm 1},
    Yumin Xu\textsuperscript{\rm 1},
    Ruixun Zhang\textsuperscript{\rm 1,\rm 2,\rm 3,\rm 4}\thanks{Corresponding author}
}
\affiliations{

    \textsuperscript{\rm 1}School of Mathematical Sciences, Peking University \\
    \textsuperscript{\rm 2}Center for Statistical Science, Peking University\\
    \textsuperscript{\rm 3}National Engineering Laboratory for Big Data Analysis and Applications, Peking University \\
    \textsuperscript{\rm 4}Laboratory for Mathematical Economics and Quantitative Finance, Peking University

    chenzixi22@stu.pku.edu.cn, xuyumin@pku.edu.cn, zhangruixun@pku.edu.cn
}

\usepackage{bibentry}



\begin{document}

\maketitle

\begin{abstract}
The nonlinear two-time-scale stochastic approximation is widely studied under conditions of bounded variances in noise. Motivated by recent advances that allow for variability linked to the current state or time, we consider state- and time-dependent noises. We show that the Lyapunov function exhibits polynomial convergence rates in both cases, with the rate of polynomial delay depending on the parameters of state- or time-dependent noises. Notably, if the state noise parameters fully approach their limiting value, the Lyapunov function achieves an exponential convergence rate. We provide two numerical examples to illustrate our theoretical findings in the context of stochastic gradient descent with Polyak-Ruppert averaging and stochastic bilevel optimization.
\end{abstract}

\section{Introduction}
Stochastic approximation (SA) was initially introduced by \citet{robbins1951stochastic} to find the root of an unobserved function operator. Specifically, it aims to approximate the solution of $f(x) = 0$ under noisy observations $f( x_k ) + \xi_k$. This method has been widely adopted in different areas including stochastic optimization, machine learning, and reinforcement learning (RL); see, for example, \citet{sutton2008convergent, Ermoliev2009, borkar2009stochastic, sutton2009fast, pham2010new, bottou2018optimization, xu2019two, lan2020first, wu2020finite, khodadadian2022finite}.

Traditional single-time-scale SA faces challenges in complex applications, and the two-time-scale SA was proposed by \citet{ermoliev1983stochastic} and developed by \citet{borkar1997stochastic} to manage complex processes with both linear and nonlinear update functions. We define the fast-scale as $x_k \in \mathbb{R}^{d_1} $ and the slow-scale as $y_k \in \mathbb{R}^{d_2}$, with unknown update functions $ f : \mathbb{R}^{d_1} \times \mathbb{R}^{d_2} \to \mathbb{R}^{d_1} $ and $ g : \mathbb{R}^{d_1} \times \mathbb{R}^{d_2} \to \mathbb{R}^{d_2} $. $f( x_k , y_k ) + \xi_k$ and $ g( x_k ,y_k ) + \psi_k$ represent their noisy observations, respectively. We consider the following two-time-scale process:
\begin{equation}
\begin{aligned}
     x_{k+1}-x_k = - \alpha_k [ f( x_k , y_k ) + \xi_k ], \\
     y_{k+1}-y_k = - \beta_k [ g( x_k ,y_k ) + \psi_k],
     \label{1}
\end{aligned}
\end{equation}
where $\{\xi_k\}$ and $\{\psi_k\}$ are martingale differences with decreasing sequences of $\alpha_k$ and $\beta_k$. It fits a dynamical system:

\begin{equation}
\begin{aligned}
    \dot{x} &= -f(x,y), \\
    \dot{y} &= -g(x,y).
\end{aligned}
\end{equation}

The goal of a stochastic two-time-scale approximation is to control the sequence so that it converges to a goal $(x^*, y^*)$. Because the two variables $x_k$ and $y_k$ interact with each other, their convergence behavior differs from that of a single variable. Typically, to ensure the convergence across both fast and slow scales, the learning rates $\alpha_k$ and $\beta_k$ are chosen to satisfy $\alpha_k \gg \beta_k$ \citep{konda2004convergence}. Under certain assumptions, they are set to a special polynomial order of magnitude, as discussed in \citet{doan2022nonlinear}. With stricter conditions on functions $f$ and $g$, the ratio of $\alpha_k$ and $\beta_k$ can be fixed, leading to convergence degenerating into a single-time-scale scenario \citep{shen2022single}.

It is subtle to determine the relationship between two timescales and their learning rates $\alpha_k$ and $\beta_k$. Our work seeks an explicit relationship between the choice of order of magnitude under certain conditions. Moreover, in many applications, the stochastic term is not bounded. 
We achieve a faster convergence rate ranging from $\mathcal{O} \left( k^{-2/3} \right)$ to $\mathcal{O} \left( k^{-\infty} \right)$---and even an exponential rate of $\mathcal{O} \left( e^{-\epsilon k} \right)$ in certain cases---with different state noise parameters $\delta_{ij}$ and time noise parameters $\gamma_i$. Here $\mathcal{O} (\cdot)$ indicates that the term is dominated by $c \times \cdot$ with constant $c$ as $k \to \infty$.

\subsection{Motivating Applications}
Two-time-scale SA is widely adopted in many algorithms because it can precisely describe the relation between variables and achieve fast convergence. We are motivated by two algorithms, stochastic gradient descent (SGD) with Polyak-Ruppert averaging and stochastic bilevel optimization.

\paragraph{SGD with Polyak-Ruppert averaging.}
In order to minimize a function $f$ with only noisy observations of its gradient, \citet{ruppert1988efficient} and \citet{polyak1992acceleration} introduced SGD with Polyak-Ruppert averaging where
\begin{equation}
\begin{aligned}
\label{eq:SGD_PR_averaging}
    x_{k+1}-x_k &= - \alpha_k ( \nabla F( x_k ) + \xi_k ), \\
    y_{k+1}-y_k &= - \beta_k ( y_k - x_k ).
\end{aligned}
\end{equation}
This algorithm has the form of two-time-scale SA (\ref{1}) with $f( x , y ) = \nabla F( x )$ and $ g(x,y) = y-x$.

\paragraph{Stochastic bilevel optimization.}
\citet{colson2007overview} provides a review of bilevel optimization, which focus on solving
\begin{equation}
\begin{aligned}
    \min_{y \in \mathbb{R}^{d_2}} l(y) := G(x^*(y),y) \\ \text{s.t.} \quad x^*(y) := \underset{x \in \mathbb{R}^{d_1}}{\argmin} F(x,y),
\end{aligned}
\end{equation}
where $F(x,y)$ and $l(x)$ are called inner and outer objective functions, respectively. With only noisy observations of gradients and Hessian matrices, \citet{shen2022single} apply the two-time-scale (\ref{1}) where
\begin{equation}
\begin{aligned}
\label{eq:SBO}
    f(x,y) &= \nabla_x F(x,y), \\
    g(x,y) &= \nabla_y G(x,y) \\
    &- \nabla_{yx}^2 F(x, y) [\nabla_{xx}^2 F(x, y)]^{-1} \nabla_x G(x,y).
\end{aligned}
\end{equation}
To accelerate their convergence rate, we consider the noisy term under various conditions. For more details of the experiment, see Section \ref{sec:numerical_experiment}.

\subsection{Related Works}
Two-time-scale SA convergence was initially introduced by \citet{borkar1997stochastic}. The convergence rate in the linear case has been established by \citet{konda2004convergence}, where $\beta_k^{-1/2}(y_k - y^*)$ converges to a normal distribution. \citet{kaledin2020finite} established a convergence rate with both martingale and Markovian noises. In these studies, the linear two-time-scale SA achieves a convergence rate of $\mathcal{O} \left( k^{-1} \right)$.

For the nonlinear case that we focus on, \citet{doan2022nonlinear} sets a baseline convergence rate of $\mathcal{O} \left( k^{-2/3} \right)$ when $f$ has a strongly monotonic property. A stronger smoothness condition leads to a single-time-scale convergence rate of $\mathcal{O} \left( k^{-1} \right)$, as shown in \citet{shen2022single}. Recently, \citet{hu2024central} established a central limit theorem for convergence with Markov noises, while \citet{han2024finite} investigated the case with local linearity, achieving convergence rates of $\mathcal{O} \left( \alpha_k \right)$ and $\mathcal{O} \left( \beta_k \right)$. \citet{doan2024fast} studied a variant in which $f$ and $g$ have a special form, also achieving a convergence rate of $\mathcal{O} \left( k^{-1} \right)$. 

In cases without strong monotonicity, \citet{shen2022single} and \citet{zeng2024two} show different convergence rates under certain smoothness conditions. 

In the field of SA, \citet{ilandarideva2023accelerated} and \citet{karandikar2023convergence} recently introduced a general noise assumption termed `state-dependent noise'. We focus on specific cases of their assumptions, which are connected with the `overparametrized model' discussed in \citet{sebbouh2021almost}. In this particular case, $\mathbb{E}[\Vert \xi_k \Vert ^2 ] \to 0$ as $x_k \to x^*$, indicating that the model is overparametrized and has no uncertainty.

\subsection{Contribution}
By requiring that martingale differences $\{\xi_k\}$ and $\{\psi_k\}$ have decreasing variances with respect to state $x_k$ and $y_k$ with parameters $\delta_{ij}$, or time $k$ with parameters $\gamma_i$, we are able to improve the convergence rate of $\mathcal{O} \left( k^{-2/3} \right)$ in \citet{doan2022nonlinear} and $\mathcal{O} \left( k^{-1} \right)$ in \citet{shen2022single}. 

More importantly, our work provides novel insights into understanding the gap between stochastic and deterministic two-time-scale approximations. We show explicitly how the convergence rate depends on the parameters $\delta_{ij}$ or $\gamma_i$ by introducing a linear programming form $m(x)$. In the proof, we propose a novel technique based on Bernoulli's inequality for balancing coefficients under different parameters. Surprisingly, we find that there could be a convergence rate of $\mathcal{O} \left( k^{-\infty} \right)$ for both parameters, and even an exponential rate of $\mathcal{O} \left( e^{-\epsilon k} \right)$ when the parameter $\delta_{ij}$ is equal to 1. 

\subsection{Outline}
Section \ref{sec:problem_setup} formulates the problem and defines notations. Sections \ref{sec:convergence_state} and \ref{sec:convergence_time} demonstrate theoretical results on convergence rates under state-dependent and time-dependent noises, respectively. Section \ref{sec:numerical_experiment} conducts numerical experiments. Section \ref{sec:conclusion} concludes. The supplementary material provides proofs of all theoretical results (Appendix A) and figures of numerical experiments (Appendix B).

\section{Problem Setup}
\label{sec:problem_setup}

In this section, we introduce several assumptions and establish several basic lemmas for convergence analysis. Generally, these assumptions are set to ensure appropriate propositions of update functions and steady convergence of the two-time-scale SA. Assumptions \ref{asm:lambda}--\ref{asm:g} below are commonly used in previous research \citep{doan2022nonlinear, shen2022single}. However, Assumptions \ref{asm:xi-psi}--\ref{asm:xi-psi_3} introduced here shed light on a new aspect of the noise term, which will be discussed later. 

\subsection{Assumptions}

We assume that there is a unique solution $(x^*, y^*)$ s.t.
\begin{equation}
\begin{aligned}
    f(x^*,y^*)&= 0, \\
    g(x^*,y^*)&= 0.
\end{aligned}
\end{equation}

\begin{assumption}
\label{asm:lambda}
Given $y$ there exists an operator $\lambda$ such that $x=\lambda(y)$ is the unique solution of
\begin{align}
    f(\lambda(y), y)=0.
\end{align}
Suppose $\lambda$ is Lipschitz continuous with respect to constant $L_\lambda$,
\begin{align}
    \Vert \lambda(y_1) - \lambda(y_2) \Vert \le L_\lambda \Vert y_1 - y_2 \Vert.
\end{align}
\end{assumption}

Given a fixed $y$, this assumption guarantees a unique equilibrium solution $\lambda(y)$ of $x$, which is the unique update target of $x$. Moreover, the Lipschitz smoothness of $\lambda$ ensures that the update target of $x$ does not change significantly when $y$ changes. These two properties of $\lambda(y)$ will be key when dealing with the differences among $x$, $\lambda(y)$, and $x^*$.

\begin{assumption}
\label{asm:f}
$f$ is Lipschitz continuous with positive constant $L_f$, i.e. $\forall x_1, x_2, y_1$, and $y_2$,
\begin{align}
    \Vert f( x_1 ,y_1) - f( x_2 , y_2) \Vert &\le L_f (\Vert x_1 - x_2 \Vert + \Vert y_1 - y_2 \Vert),
\end{align}
and $f$ is strongly monotone w.r.t $x$ when $y$ is fixed, i.e. there exists a constant $\mu_f > 0$,
\begin{align}
    \langle x_1 - x_2, f(x_1, y) - f(x_2, y) \rangle \ge \mu_f \Vert x_1 - x_2 \Vert^2.
\end{align}
\end{assumption}

With Lipschitz continuity and strong monotonicity of $f$, the update of $x$ is controlled to be close to its target. This can be replaced by weaker requirements when only considering linear two-time-scale SA convergence, as shown in \citet{konda2004convergence} and \citet{kaledin2020finite}.

\begin{assumption}
\label{asm:g}
The operator $g(\cdot, \cdot)$ is Lipschitz continuous with constant $L_g$, i.e. $\forall x_1, x_2, y_1$, and $y_2$,
\begin{align}
    \Vert g(x_1,y_1) - g(x_2, y_2) \Vert \le L_g ( \Vert x_1 - x_2 \Vert + \Vert y_1 - y_2 \Vert).
\end{align}
Moreover, $g$ is 1-point strongly monotone w.r.t $y^*$, i.e. there exists a constant $\mu_g > 0$ such that
\begin{align}
    \langle y - y^*, g(\lambda(y), y) \rangle &\ge \mu_g \Vert y - y^* \Vert^2.
\end{align}
\end{assumption} 

For $g$, similar requirements are necessary to guarantee the convergence of $y$, where the strong monotonicity of $g$ is directly related to the final solution $y^*$. This condition is widely used in the literature on nonlinear two-time-scale studies due to the non-linear nature of the problems involved. Note that strong monotonicity is a weaker condition than strong convexity and covers more application scenarios.

To better illustrate the update process, we introduce two residual variables:
\begin{align}
    \hat{x}_k &= x_k - \lambda(y_k),\\
    \hat{y}_k &= y_k - y^*.
\end{align}
These variables represent the differences between $x_k$, $y_k$, and their respective update targets $\lambda(y_k)$ and $y^*$. We will show that the residual variables decrease to $0$ as $k \to +\infty$ when investigating the Lyapunov function.

In the context of state- or time-dependent SA, the variance of the stochastic term can be a function of state or time \citep{karandikar2023convergence}, rather than being simply bounded. We formulate the following assumptions and introduce parameters $\delta_{ij}$ and $\gamma_i$.

\begin{assumption}
\label{asm:xi-psi}
We denote by $\mathcal{Q}_k$ the filtration containing all the history generated by up to the iteration $k$, i.e.,
\begin{align}
    \mathcal{Q}_k = \{ x_0, y_0, \xi_0, \psi_0, \xi_1, \psi_1, \ldots, \xi_k, \psi_k\} ,
    \label{12}
\end{align}
where $\{\xi_k\}_{k \ge 0}$ are independent random variables with zero mean and bounded variances a.s., and so $\{\psi_k\}_{k \ge 0}$ are. Specifically, there exist $\Gamma_{11}$ and $\Gamma_{22}$ such that variances can be controlled as follows almost surely,
\begin{align}
    \mathbb{E}[\Vert \xi_k \Vert ^2 | \mathcal{Q}_{k-1}] &\le \Gamma_{11} \Vert \hat{x}_k \Vert ^ {2 \delta_{11}} + \Gamma_{12} \Vert \hat{y}_k \Vert ^ {2 \delta_{12}} , \\
    \mathbb{E}[\Vert \psi_k \Vert ^2 | \mathcal{Q}_{k-1}] &\le \Gamma_{21} \Vert \hat{x}_k \Vert ^ {2 \delta_{21}} + \Gamma_{22} \Vert \hat{y}_k \Vert ^ {2 \delta_{22}} .
\end{align}
where $\delta_{11}, \delta_{12}, \delta_{21}$, and $\delta_{22}$ are in $[0, 1)$. Meanwhile, denote $\Delta_{ij} = 1 - \delta_{ij}$ for all $i,j \in \{1, 2\}$.
\end{assumption}

\begin{assumption}
\label{asm:xi-psi_2}
$\{\xi_k\}_{k \ge 0}$ and $\{\psi_k\}_{k \ge 0}$ are martingale sequences similarly defined in Assumption \ref{asm:xi-psi} with $\delta_{ij} \equiv 1$, in other words a.s.,
\begin{align}
    \mathbb{E}[\Vert \xi_k \Vert ^2 | \mathcal{Q}_{k-1}] &\le \Gamma_{11} \Vert \hat{x}_k \Vert ^ 2 + \Gamma_{12} \Vert \hat{y}_k \Vert ^ 2, \\
    \mathbb{E}[\Vert \psi_k \Vert ^2 | \mathcal{Q}_{k-1}] &\le \Gamma_{21} \Vert \hat{x}_k \Vert ^ 2 + \Gamma_{22} \Vert \hat{y}_k \Vert ^ 2 .
\end{align}
\end{assumption}

Here, we introduce four variables $\{ \delta_{ij} \}$ to demonstrate the relationship between the variances of stochastic terms and two errors. If $\delta_{ij} \equiv 0$, it reduces to the form already studied in \citet{doan2022nonlinear}. However, when the stochastic term is generated from a more general form of stochastic observation, we will have $\delta_{ij} > 0$, which are the so-called state-dependent noise assumptions introduced by \citet{ilandarideva2023accelerated} and \citet{karandikar2023convergence}. Assumptions \ref{asm:xi-psi} and \ref{asm:xi-psi_2} introduce more fundamental assumptions of an overparametrized model studied in \citet{sebbouh2021almost} with respect to two-time-scale convergence. Examples are widely used in Robbins-Monro setting SA, least squares, logistic regression \citep{moulines2011non}, and articles on stochastic gradient descent \citep{chen2020statistical}. 

With $\delta_{ij} > 0$, when $x_k$ and $y_k$ are closer to their targets, the variance of the stochastic term will be lower. Intuitively, this can accelerate the convergence rate compared to the case where $\delta_{ij} \equiv 0$ because the variance decreases. We will theoretically demonstrate that these assumptions accelerate the convergence rate. Additionally, there is a qualitative acceleration when $\delta_{ij} \equiv 1$, where Assumption \ref{asm:xi-psi_2} is sufficient to achieve an exponential convergence rate of $\mathcal{O} \left( e^{-\epsilon k} \right)$. 

It should be noted that we only make separate assumptions about the variance for both $\{ \xi_k \}$ and $\{ \psi_k \}$, respectively. However, we do not impose any limits on the correlation between the two martingale differences.

For simplicity, we only consider the condition under a certain form of coefficients $\alpha_k$ and $\beta_k$, where
\begin{align}
\label{form_of_alpha}
    \alpha_k = \frac{\alpha}{(k + 1 + k_0)^a}, \beta_k = \frac{\beta}{(k + 1 + k_0)^b}.
\end{align}

\begin{assumption}
\label{asm:xi-psi_3}
$\{\xi_k\}_{k \ge 0}$ and $\{\psi_k\}_{k \ge 0}$ are similarly defined in (\ref{12}), and for a fixed $k_0$ there exist constants $\Gamma_{11}', \Gamma_{22}', \gamma_1, \gamma_2 \ge 0$ such that $\gamma_1 - \gamma_2 \in [-1, \frac{1}{2}) $ and variances can be controlled as follows almost surely,
\begin{align}
    \mathbb{E}[\Vert \xi_k \Vert ^2 | \mathcal{Q}_{k-1}] &\le \Gamma_{11}' (k+1+k_0)^{-\gamma_1}, \\
    \mathbb{E}[\Vert \psi_k \Vert ^2 | \mathcal{Q}_{k-1}] &\le \Gamma_{22}' (k+1+k_0)^{-\gamma_2} .
\end{align}
\end{assumption}

In contrast to the former definition, this could be regarded as time-dependent noise assumptions. It represents a special case of the conditions studied in \citet{karandikar2023convergence}. We will also study how the convergence rate changes with respect to $\gamma_i$. Similarly, as iteration $k$ increases, the variance of the update will decrease, potentially accelerating the convergence rate. However, as $\gamma_i \to \infty$, it will not achieve an exponential convergence rate. This reflects an inherent difference between state-independent and time-dependent two-time-scale SA. Note that the condition $\gamma_1 - \gamma_2 \in [-1, \frac{1}{2}) $ is necessary in the proof of Theorem \ref{thm:3} to balance the two different convergence rates.

\subsection{Connection Between Assumptions and Application}

Assumptions \ref{asm:xi-psi_2}--\ref{asm:xi-psi_3}, along with the exponential convergence rate, are intended to demonstrate a specific scenario where fast convergence is achievable and valuable in practical applications of stochastic approximation. Although these assumptions may appear restrictive, they are motivated by and have strong connections with control systems and reinforcement learning, where exponential convergence is essential for efficient operations. For example, recent studies such as \citet{chen2020statistical} incorporate a second-order moment assumption for SGD analysis, and \citet{kaledin2020finite} highlights how contraction-based assumptions can accelerate convergence, even in multi-scale frameworks.

Moreover, Assumption \ref{asm:xi-psi_2} supports various applications in non-linear optimization and two-time-scale stochastic control, where stringent convergence rates facilitate computational efficiency. Practical examples include high-frequency trading algorithms \citet{abergel2016limit, gatheral2011volatility} and adaptive filtering \citet{haykin2002adaptive}, where stability and rapid convergence are critical to performance. Furthermore, recent work on stochastic gradient methods \citet{bottou2018optimization} underscores the importance of exponential rates in improving the training efficiency of complex models. By rigorously analyzing this case, we provide a theoretical benchmark that could inform future research in settings where fast convergence is feasible and advantageous.

\subsection{Lemmas}

In general, we have
\begin{align}
    \hat{x}_{k+1} &= \hat{x}_k - \alpha_k f(x_k, y_k) + \lambda(y_k) - \lambda(y_{k+1}) - \alpha_k \xi_k, \\
    \hat{y}_{k+1} &= \hat{y}_k - \beta_k g(\lambda(y_k), y_k) \notag \\
    &+ \beta_k ( g(\lambda(y_k), y_k) - g(x_k, y_k) ) - \beta_k \psi_k .
\end{align}

Here, we show the upper bounds for $\hat{x}_k$ and $\hat{y}_k$. Prior to this, we propose the following lemmas to provide the recursive relational formulas for errors. Here we denote $a \lesssim b$ when there exists an irrelevant constant $c$ such that $a \le cb$.

\begin{lemma}
\label{lem:1}
Suppose that Assumptions \ref{asm:lambda}--\ref{asm:xi-psi} hold. Let $x_k$ and $y_k$ be updated by (\ref{1}). Then, for all $k \ge 0$, we have
\begin{align}
        &\mathbb{E} [\Vert \hat{x}_{k+1} \Vert ^2 | \mathcal{Q}_{k-1}] - (1 - \mu_f \alpha_k) \Vert \hat{x}_k \Vert ^2 \notag \\
        \lesssim& (\beta_k + \alpha_k^2) \Vert \hat{x}_k \Vert ^2 + (\beta_k + \beta_k^2) \Vert \hat{y}_k \Vert ^2 \notag \\
        &+ \frac{\beta_k^2}{\alpha_k} \mathbb{E}[\Vert \psi_k \Vert^2 | \mathcal{Q}_{k-1}] + \alpha_k^2 \mathbb{E}[\Vert \xi_k \Vert^2 | \mathcal{Q}_{k-1}].
\end{align}
\end{lemma}
\begin{proof}
    We provide a detailed statement and proof in Appendix A.1 of the supplementary material.
\end{proof}

\begin{lemma}
\label{lem:2}
Suppose that Assumptions \ref{asm:lambda}--\ref{asm:xi-psi} hold. Let $x_k$ and $y_k$ be generated by (\ref{1}). Then, for all $k \ge 0$, we have
\begin{align}
    &\mathbb{E} [\Vert \hat{y}_{k+1} \Vert ^2 |\mathcal{Q}_{k-1}] - (1 - \mu_g \beta_k ) \Vert \hat{y}_k \Vert ^2 \notag \\
    \lesssim& (\beta_k + \beta_k^2) \Vert \hat{x}_k \Vert ^2 + \beta_k^2 \Vert \hat{y}_k \Vert ^2 +  \beta_k^2 \mathbb{E} [\Vert \psi_k \Vert ^2 | \mathcal{Q}_{k-1}].
\end{align}
\end{lemma}
\begin{proof}
    We provide a detailed statement and proof in Appendix A.2 of the supplementary material.
\end{proof}

These two lemmas provide a well-controlled upper bound on the expected error between two adjacent iterations, establishing our basic foundation. We also observe that the two main terms that affect the convergence rate are $\alpha_k \Vert \hat{x}_k \Vert ^2$ and $\beta_k \Vert \hat{y}_k \Vert ^2$. These terms play an important role when considering the Lyapunov function stated below.

We define $c = 4\frac{L_g^2}{\mu_f \mu_g}$ and the Lyapunov function here, denote it as $V_k$, is:
\begin{align}
    V_k = V(\hat{x}_k, \hat{y}_k) = c \frac{\beta_k}{\alpha_k} \Vert \hat{x}_k \Vert^2 + \Vert \hat{y}_k \Vert^2.
\end{align}
An insightful observation is that we need to balance two different variables, $x_k$ and $y_k$. This will lead to the aforementioned Lyapunov function $V_k$.

Combining the above two lemmas allows us to analyze the update of the Lyapunov function in the next lemma. It is important to note that the difference between two iterations depends primarily on the magnitude of $\beta_k$ and the distinct forms of $\mathbb{E}[\Vert \xi_k \Vert^2]$ and $\mathbb{E} [\Vert \psi_k \Vert ^2]$.

\begin{lemma}
\label{lem:3}
    Suppose that Assumptions \ref{asm:lambda}--\ref{asm:xi-psi} hold. Let $x_k$ and $y_k$ be generated by (\ref{1}). Suppose $\alpha_k$ and $\beta_k$ decrease (not necessarily strictly), and $\frac{\alpha_k}{\beta_k}$ is sufficiently large. With a constant $C(\alpha_0, \beta_0)$ generated by $\alpha_0$ and $\beta_0$, we have
    \begin{align}
        &\mathbb{E} [ V_{k+1} ] - ( 1 - \frac{1}{2} \mu_g \beta_k ) \mathbb{E} [ V_k ] \lesssim  C(\alpha_0, \beta_0) \alpha_k^2 \mathbb{E} [ V_k ] \notag \\
        &+ \alpha_k \beta_k \mathbb{E}[\Vert \xi_k \Vert^2] + \left( \frac{\beta_k^3}{\alpha_k^2} + \beta_k^2 \right) \mathbb{E} [\Vert \psi_k \Vert ^2].
    \end{align}
\end{lemma}
\begin{proof}
    We provide a detailed statement and proof in Appendix A.3 of the supplementary material.
\end{proof}

\section{Convergence Under State-Dependent Noise Assumptions}
\label{sec:convergence_state}

Now we consider the convergence results under certain forms of coefficients $\alpha_k$ and $\beta_k$ as in (\ref{form_of_alpha}). These step sizes form an algebraic fraction with respect to $k$ with a translation to ensure a moderate update rate. Here, we denote $a \land b = \min \{ a, b \}$ and $a \lor b = \max \{ a, b \}$. In Assumption \ref{asm:xi-psi}, we define a linear programming function
\begin{align}
    m(x) =& \frac{(1+\delta_{11}) x - \delta_{11}}{1 - \delta_{11}} \land \frac{x}{1 - \delta_{12}} \notag \\
    &\land \frac{(2 - \delta_{21}) (1-x) }{1 - \delta_{21}} \land \frac{2(1-x)}{ 1 - \delta_{22}} ,
\end{align}
which is set to be the minimum of four linear functions. $m(x)$ depending on all four parameters $\delta_{ij}$ is a delicate function set to satisfy
\begin{align}
    & -1 + 2a \notag \\
    \ge& -1+a+t+(1-a-t)\delta_{11} \lor -1+a+t-t\delta_{12} \notag \\
    &\lor -3+4a+t+(1-a-t)\delta_{21} \lor -3+4a+t-t\delta_{22},
\end{align}
where $a=\argmax_{x \in (\frac{1}{2}, 1]} m(x)$ and $t = \max_{x \in (\frac{1}{2}, 1]} m(x)$. This balances all terms with different orders in the proof and achieves the maximum convergence rate of $\mathcal{O}(k^{-t})$. Here, we draw the following conclusion.

\begin{thm}
\label{thm:1}
    Suppose that Assumptions \ref{asm:lambda}--\ref{asm:xi-psi} hold. Let $x_k$ and $y_k$ be generated by (\ref{1}). We assume
    \begin{gather}
        b=1 , a=\argmax_{x \in (\frac{1}{2}, 1]} m(x) , t = \max_{x \in (\frac{1}{2}, 1]} m(x). 
    \end{gather}
    In addition, $\beta$ and $\frac{\alpha}{\beta}$ are sufficiently large. Let $C_{\alpha\beta} = C(\alpha, \beta) \ge C(\alpha_0, \beta_0)$ defined in Lemma \ref{lem:3} and set $k_0$ as large enough. There exist two functions generated by $M$ with $C_1(M) = \mathcal{O} (M)$ and $C_2(M) = \mathcal{O} \left( M^{\delta_{11} \lor \delta_{12} \lor \delta_{21} \lor \delta_{22}} \right)$ as $M \to +\infty$.
    Note that $\delta_{ij} < 1$, there exists a $M$ s.t.
    \begin{gather}
        M \ge 3 k_0^t V_0 \lor \frac{3}{a} C_2(M) .
    \end{gather}
    Together, we have $\forall k \ge 0$,
    \begin{align}
        \mathbb{E} [ V_k ] \le \frac{M}{(k + k_0)^t} .
    \end{align}
\end{thm}
\begin{proof}
    We provide a detailed statement and proof in Appendix A.4 of the supplementary material.
\end{proof}

Theorem \ref{thm:1} provides a general view of how the convergence rate relates to the parameters $\delta_{ij}$ through a linear programming function $m(x)$. To better understand the implications of Theorem \ref{thm:1}, particularly regarding the behavior of $m(x)$, we provide two enlightening special cases.

\begin{corollary}
\label{cor:1}
    Suppose that Assumptions \ref{asm:lambda}--\ref{asm:xi-psi} hold. Let $x_k$ and $y_k$ be generated by (\ref{1}). Moreover, we assume $\delta_{11} = \delta_{12}$ and $\delta_{21} = \delta_{22}$. Recall that $\Delta_{ij} = 1 - \delta_{ij}$, then under the condition of Theorem \ref{thm:1}, 
    \begin{align}
        a = \frac{\Delta_{11} + \Delta_{22}}{ \Delta_{11} + 2 \Delta_{22}}, t = \frac{1 + \Delta_{22}}{\Delta_{11} + 2 \Delta_{22}}.
    \end{align}
    There exist $\alpha, \beta, k_0$, and $M$, s.t. $\forall k \ge 0$,
    \begin{align}
        \mathbb{E} [ V_k ] \le \frac{M}{(k + k_0)^t} .
    \end{align}
\end{corollary}
\begin{proof}
    We provide a detailed statement and proof in Appendix A.5 of the supplementary material.
\end{proof}

It is evident that under the assumption that $\delta_{11} = \delta_{12}$ and $\delta_{21} = \delta_{22}$, if $\delta_{11} \neq 0$, we always have $a < t$. This implies that the condition $\delta_{ij} > 0$ indeed accelerates the convergence rate to be faster than $\mathcal{O} \left( k^{-a} \right)$, which is not demonstrated in earlier studies. Furthermore, when $\delta_{ij} = 0 ( \Delta_{ij} = 1 )$, we find that $a = t = \frac{2}{3}$. This is consistent with the results in \citet{doan2022nonlinear}.

Moreover, this bridges the gap between stochastic and deterministic two-time-scale SA. The latter is well-known for achieving at least an exponential convergence rate of $\mathcal{O} \left( e^{-\epsilon k} \right)$. As $\delta_{ij} \to 1^- ( \Delta_{ij} \to 0^+ )$, we observe that $t \to +\infty$, which naturally leads to the following result.

Now we investigate the case in Assumption \ref{asm:xi-psi_2} with constant learning rates $\alpha_k = \alpha$ and $\beta_k = \beta$. We define $c = 4\frac{L_g^2}{\mu_f \mu_g}$. The Lyapunov function is defined as follows and denoted as $V_k$,
\begin{align}
    V_k = V(\hat{x}_k, \hat{y}_k) = c \frac{\beta}{\alpha} \Vert \hat{x}_k \Vert^2 + \Vert \hat{y}_k \Vert^2.
\end{align}

\begin{thm}
\label{thm:2}
    Suppose that Assumptions \ref{asm:lambda}--\ref{asm:g},\ref{asm:xi-psi_2} hold. Let $x_k$ and $y_k$ be generated by (\ref{1}). Let $C_\beta = B_1 + B_2 \beta$ with two constants $B_1, B_2$, and we set $\omega = \frac{\alpha}{\beta}$. There exist three functions generated by $\omega$ with $D_1 (\omega) = - \frac{1}{2} \mu_g + \mathcal{O} \left( \frac{1}{\omega} \right)$ as $\omega \to +\infty$, $D_2 (\omega)$, and $D_3 (\omega)$.
    We assume $\omega$ to be sufficiently large and
    \begin{gather}
        D_1 (\omega) \le - \frac{1}{4} \mu_g.
    \end{gather}
    Then there exists a $\beta > 0$ s.t.
    \begin{align}
        e^{-\epsilon} \triangleq 1 + D_1 (\omega) \beta + D_2 (\omega) \beta^2 + D_3 (\omega) \beta^3 < 1.
    \end{align}
    Together, we have $\forall k \ge 0$,
    \begin{align}
        \mathbb{E} [ V_k ] \le e^{-\epsilon k} V_0 .
    \end{align}
\end{thm}
\begin{proof}
    We provide a detailed statement and proof in Appendix A.6 of the supplementary material.
\end{proof}

Combining Theorem \ref{thm:1} and \ref{thm:2}, it is now clear that the boundary between the polynomial convergence rate and the exponential convergence rate occurs at $\delta_{ij} \equiv 1$. Specifically, if $\delta_{ij} < 1$, the convergence rate remains polynomial, but as $\delta_{ij} \to 1^-$, the convergence rate increases significantly. When $\delta_{ij} \equiv 1$, an exponential convergence rate is achieved.

\section{Convergence Under Time-Dependent Noise Assumptions}\label{sec:convergence_time}

We continue with the specified form of the coefficients $\alpha_k$ and $\beta_k$ as in (\ref{form_of_alpha}),
and proceed to derive our conclusion under Assumption \ref{asm:xi-psi_3}, where the noises satisfy the decay property at a rate of $\mathcal{O} \left( k^{-\gamma_i} \right)$.

\begin{thm}
\label{thm:3}
    Suppose that Assumptions \ref{asm:lambda}--\ref{asm:g},\ref{asm:xi-psi_3} hold. Let $x_k$ and $y_k$ be generated by (\ref{1}). We assume \begin{gather}
        b=1 , a=\frac{2-\gamma_1+\gamma_2}{3} \in (\frac{1}{2}, 1] , t = \frac{2+2\gamma_1+\gamma_2}{3}.
    \end{gather}
    In addition, $\beta$ and $\frac{\alpha}{\beta}$ are sufficiently large. Let $C_{\alpha\beta} = C(\alpha, \beta) \ge C(\alpha_0, \beta_0)$ defined in Lemma \ref{lem:3} and set $k_0$ as large enough. There exists a function generated by $M$ with $C_1(M) = \mathcal{O} (M)$ as $M \to +\infty$ and a constant $C_2'$.
    There exists a $M$ s.t.
    \begin{gather}
        M \ge 3 k_0^t V_0 \lor \frac{3}{a} C_2'.
    \end{gather}
    Together, we have $\forall k \ge 0$,
    \begin{align}
        \mathbb{E} [ V_k ] \le \frac{M}{(k + k_0)^t} .
    \end{align}
\end{thm}
\begin{proof}
    We provide a detailed statement and proof in Appendix A.7 of the supplementary material.
\end{proof}

Similarly, when $\gamma_i \equiv 0$, this reduces to the special case studied in \citet{doan2022nonlinear}. Furthermore, we observe that $t \to + \infty$ as $\gamma_i \to + \infty$, which aligns with the intuition that a higher order of the time-dependent noise parameters $\gamma_i$ also accelerates convergence. However, note that the magnitude $t$ increases linearly with respect to $\gamma_i$, which contrasts with the findings of Theorem 1. This difference explains why we cannot reach the same conclusion about the exponential convergence rate of $\mathcal{O} \left( e^{-\epsilon k} \right)$ as stated in Theorem 2.

\section{Numerical Experiment}
\label{sec:numerical_experiment}

We conduct numerical experiments using two classical examples to illustrate the convergence rates in Sections \ref{sec:convergence_state} and \ref{sec:convergence_time}. We initialize $(x_0, y_0) = (1, 1)$, and choose $(\alpha, \beta)$ to satisfy the same conditions as in Theorems \ref{thm:1}-\ref{thm:3}. The step sizes $\alpha_k$ and $\beta_k$ are set as $\alpha_k = \frac{\alpha}{(k + 1 + k_0)^a}$ and $\beta_k = \frac{\beta}{(k + 1 + k_0)^b}$, where $k_0$ is chosen optimally. It is important to note that $k_0$ increases rapidly with increasing $(\alpha, \beta)$ and affects the step sizes exponentially. Therefore, within constraints, we set the initial values of $(\alpha, \beta)$ as small as possible to prevent slow convergence caused by the excessively small initial step sizes \citep{moulines2011non}. For simplicity of analysis, we set $\delta_{ij} \equiv \delta, \Gamma_{ij} \equiv \Gamma$, and $\gamma_k \equiv \gamma$. $ \Gamma_{kk}^{'} \equiv \Gamma^{'}$ as noise parameters. All experiments are repeated 1000 times.

For $\delta \in [0, 1)$, we present log-log plots of the averaged $V_k$ with respect to the number of iterations. For $\delta = 1$, we use the logarithmic plot. The slope of the line indicates the convergence rate. Specifically, a relationship of the form $y = r x^{-s}$ corresponds to $\log y = -s \log x + \log r$. To ensure stability, we calculate the slope of the error lines using iterations after $10^6$.

\begin{figure*}[th!]
    \centering
    {
        \includegraphics[width = 0.45\linewidth]{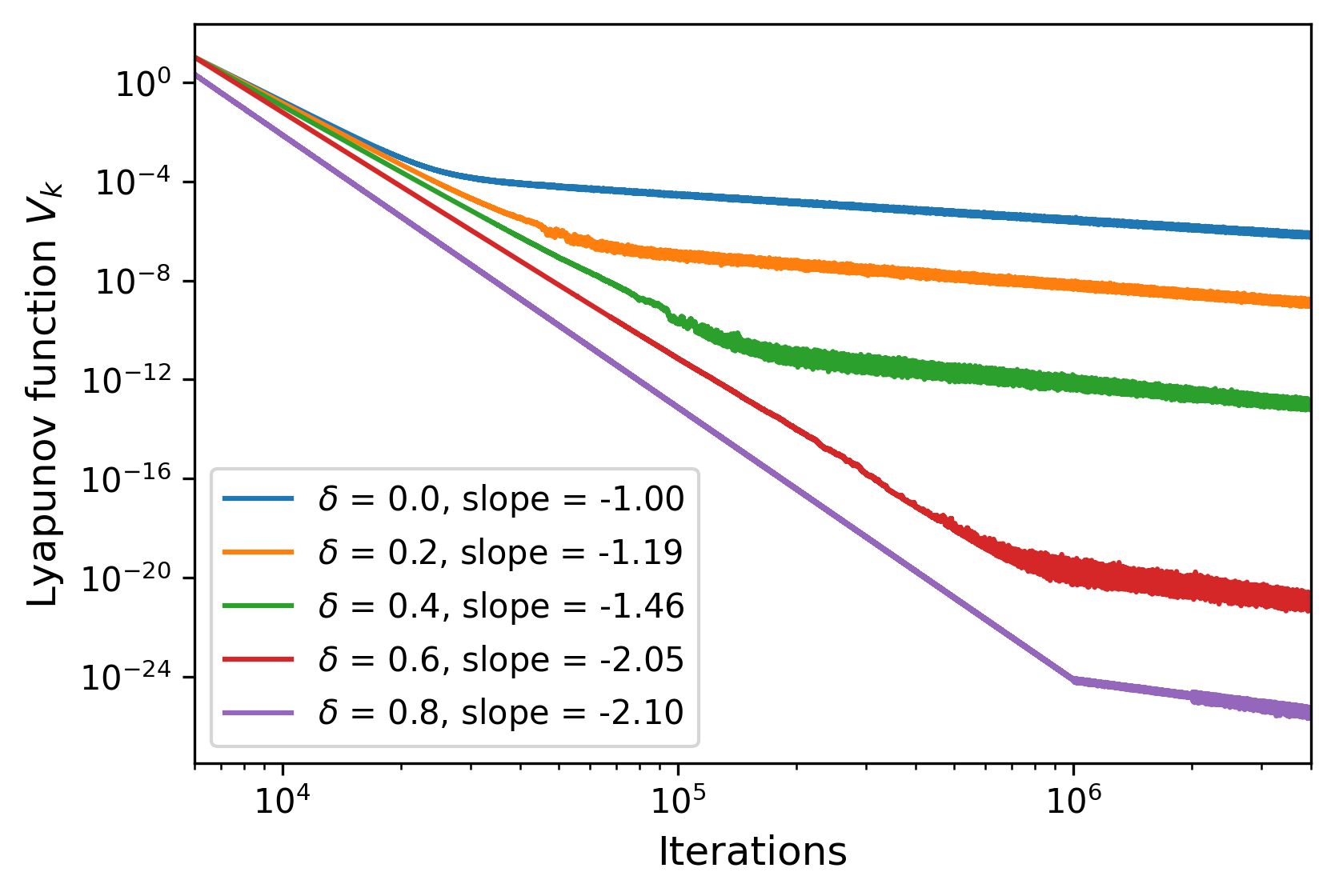}
    }
    {
        \includegraphics[width = 0.45\linewidth]{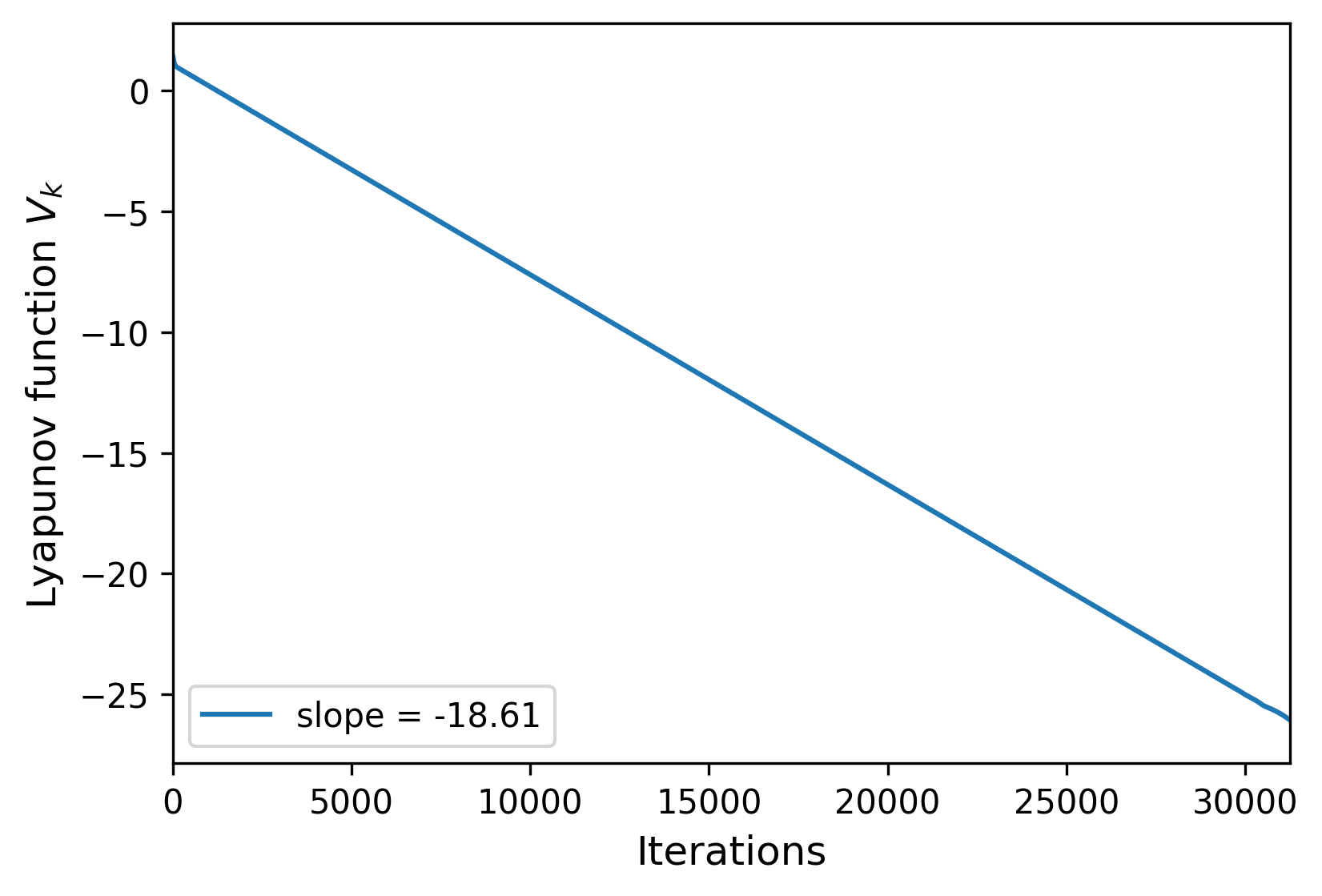}
    }
\caption{The convergence results of SGD with Polyak-Ruppert averaging. The figure on the left is a log-log plot in case $\delta=0.0, 0.2, 0.4, 0.6, 0.8$; the figure on the right is a logarithmic plot in case $\delta=1$. All R-squares of the fitted slopes do not exceed 1e-6.}
\label{fig:total_SGD_PR_averaging}
\end{figure*}

\begin{figure*}[th!]
    \centering
    {
        \includegraphics[width = 0.45\linewidth]{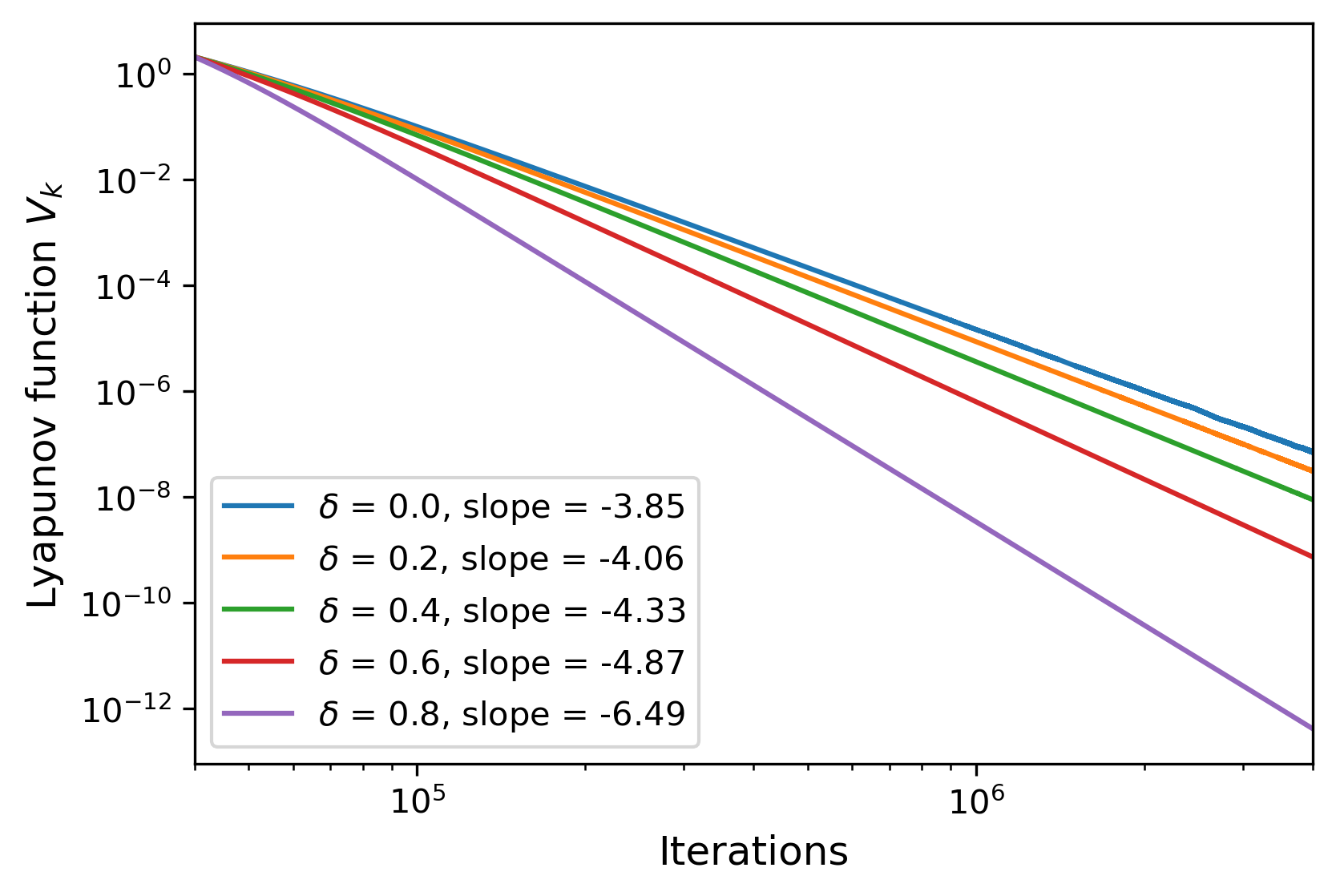}
    }
    {
        \includegraphics[width = 0.45\linewidth]{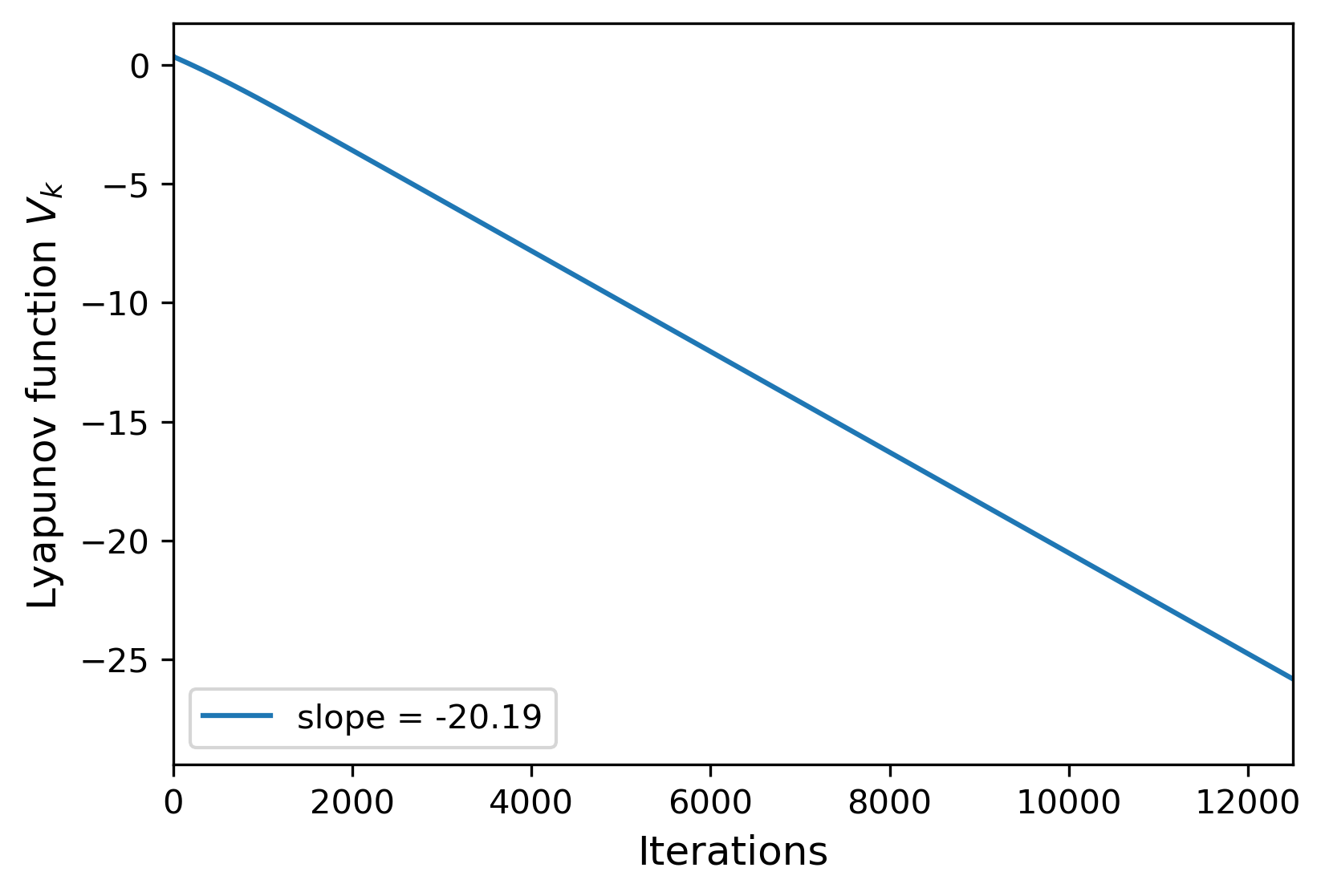}
    }
\caption{The convergence results of stochastic bilevel optimization. The figure on the left is a log-log plot in case $\delta=0.0, 0.2, 0.4, 0.6, 0.8$; the figure on the right is a logarithmic plot in case $\delta=1$. All R-squares of the fitted slopes do not exceed 0.01.}
\label{fig:total_SBO}
\end{figure*}

\paragraph{SGD with Polyak-Ruppert averaging.} 
We employ SGD with Polyak-Ruppert averaging (\ref{eq:SGD_PR_averaging}) under 5-dimensional setting to
minimize the function $F(x) = (x_1^2 + \sin x_1, \cdots, x_5^2 + \sin x_5)$ under normal white noise $\{ \xi_k \}$, which satisfies either Assumption \ref{asm:xi-psi} or \ref{asm:xi-psi_2}. The function $f( x , y ) = \nabla F( x )$ meets Assumptions \ref{asm:lambda}--\ref{asm:g} in Section \ref{sec:problem_setup}, with parameters set to $L_f = 3, \mu_f = 1, L_g = 2, \mu_g = 1$, and $L_{\lambda} = 0$.

Figure \ref{fig:total_SGD_PR_averaging} illustrates the convergence results under various parameters. For $\delta_{ij} \equiv \delta \in [0, 1)$, we examine five values: $\delta = 0.0, 0.2, 0.4, 0.6, 0.8$ with $\Gamma = 0.02$. As shown in Figure \ref{fig:total_SGD_PR_averaging}, at the beginning of iterations, the convergence rate is no slower than our theoretical results. Moreover, with increasing iterations and the decay of errors (or Lyapunov functions $V_k$), the convergence rate is also close to our theoretical bound. Specifically, the slope closely approximates $t$ as defined in Theorem \ref{thm:1}. Moreover, as $\delta$ increases, the absolute value of the slope also increases, indicating faster convergence of the algorithm. This result is consistent with the intuition that a smaller noise volatility allows faster convergence without significant fluctuations. Figure \ref{fig:total_SGD_PR_averaging} demonstrates that the algorithm actually achieves an exponential convergence rate of $\mathcal{O} \left( e^{-\epsilon k} \right)$ when $\delta = 1$ and $\Gamma^{'} = 0.1$.

\paragraph{Stochastic bilevel optimization.}
We consider an example of stochastic bilevel optimization (\ref{eq:SBO}) specified as follows:
\begin{align}
\label{example:SBO}
    F (x, y) &= 10 (x + \widetilde{h}_2(y))^2 + 10 \sin (x + \widetilde{h}_2(y)),\\ 
    G (x, y) &= (x + \widetilde{h}_2(y))^2 + \sin (y) + y^2,
\end{align}
where $\widetilde{h}_2 (z) = \operatorname{sign}(z) z^2 /2$ if $|z| \leq 1$ and $\operatorname{sign}(z) (|z| - 1/2)$ otherwise. We directly take the partial derivative of equation (\ref{example:SBO}) to obtain $f(x, y)$ and $g(x, y)$. It is straightforward to verify that Assumptions \ref{asm:lambda}--\ref{asm:g} hold with parameters set to $L_f = 60, \mu_f = 10, L_g = 3, \mu_g = 1$, and $L_\lambda = 3$. The noise terms $\{ \xi_k \}$ and $\{ \psi_k \}$ follow independent normal distributions satisfying either Assumption \ref{asm:xi-psi} or \ref{asm:xi-psi_2}.

\begin{figure*}[th!]
    \centering
    {
        \includegraphics[width = 0.45\linewidth]{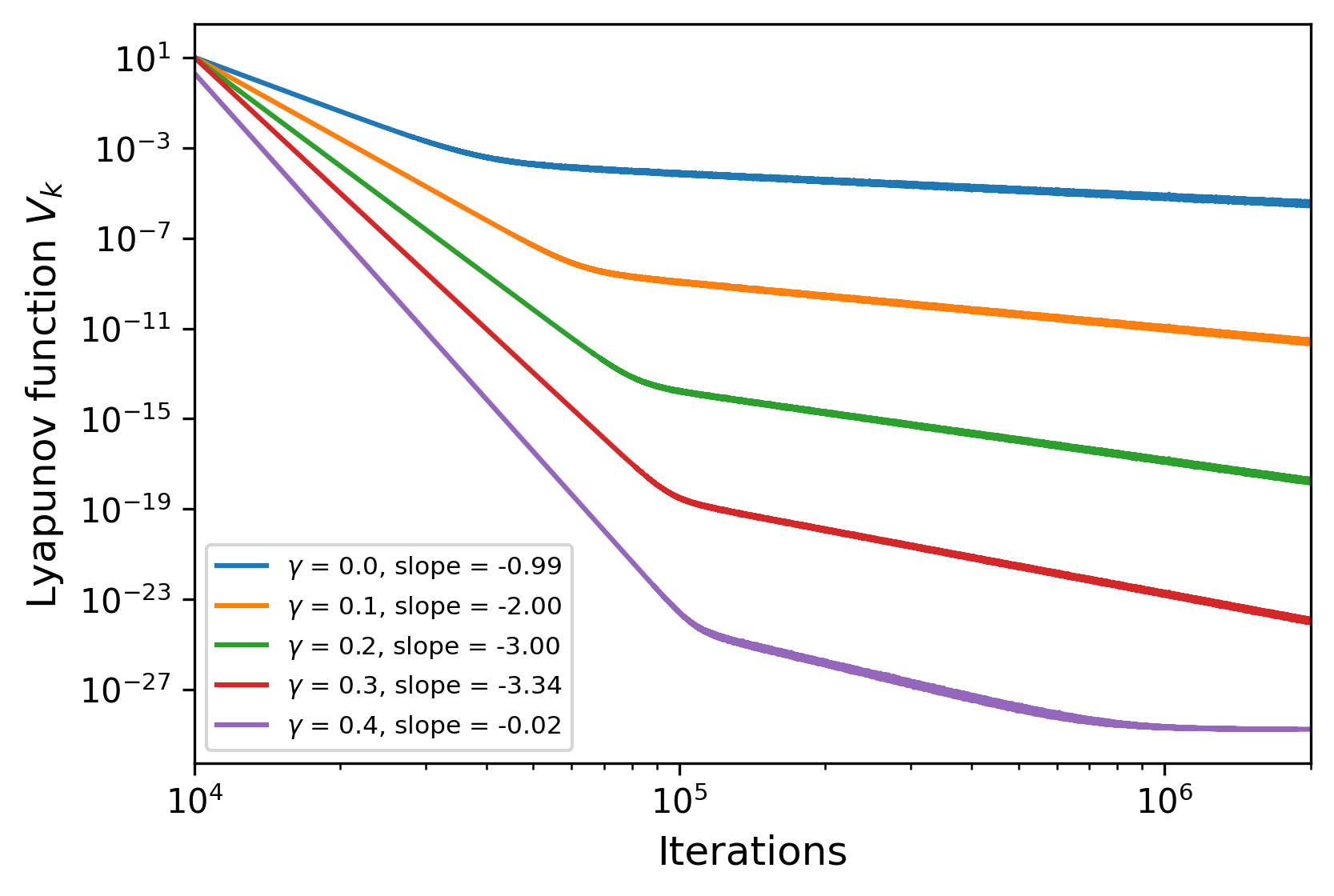}
    }
    {
        \includegraphics[width = 0.45\linewidth]{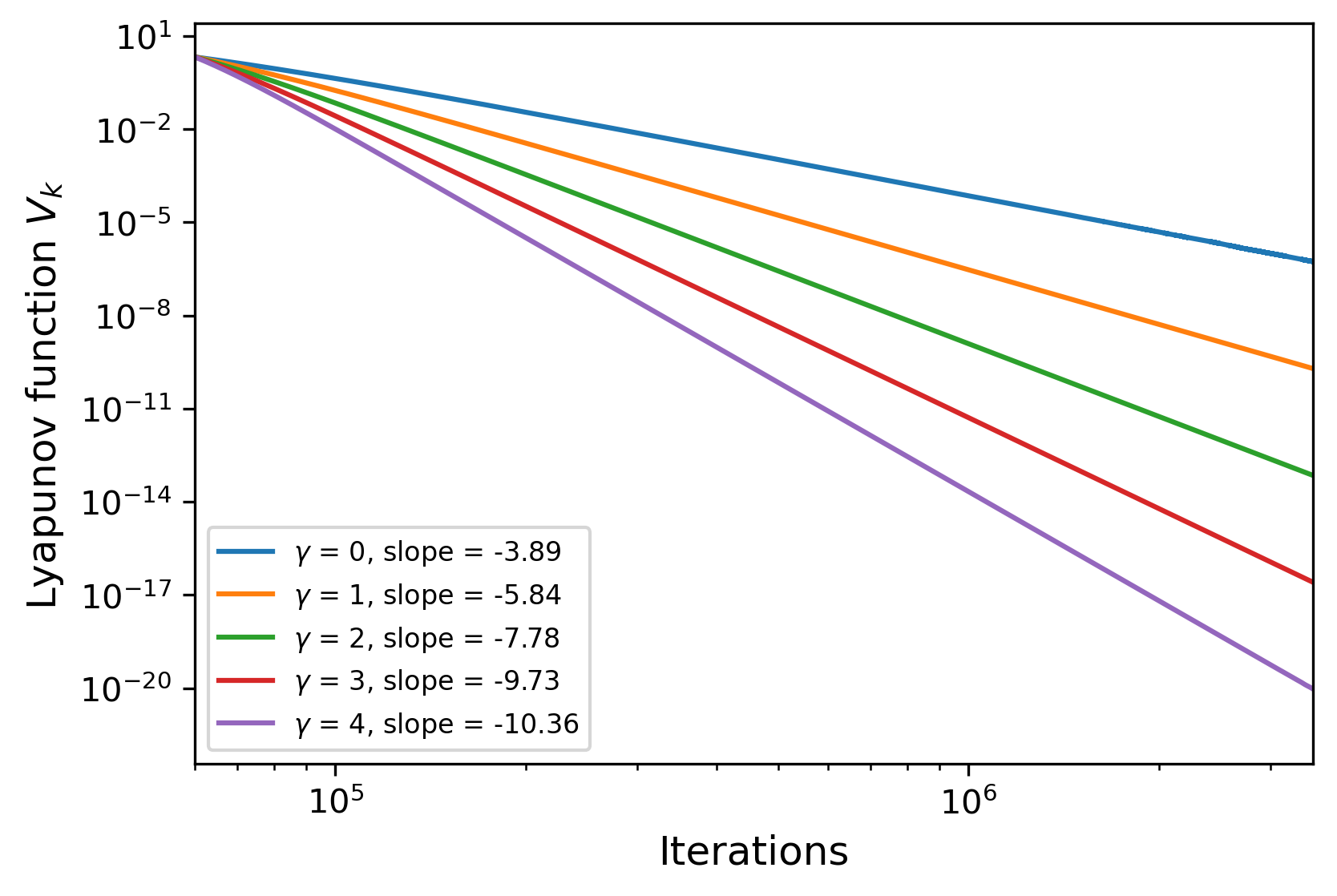}
    }
\caption{The convergence results of SGD with Polyak-Ruppert averaging (left) and stochastic bilevel optimization (right) under time-dependent noise assumption. The figure is a log-log plot in case $\gamma=0, 1, 2, 3, 4$. All R-squares of the fitted slopes do not exceed 0.1.}
\label{fig:total_time}
\end{figure*}

Figure \ref{fig:total_SBO} illustrates the convergence results of the stochastic bilevel optimization algorithm under two cases $\delta \in [0, 1)$ and $\delta = 1$. The convergence rate increases as $\delta$ approaches 1 within the interval $[0, 1)$. Notably, all experiments in this problem converge no slower than the theoretical predictions. Furthermore, the algorithm achieves an exponential convergence rate when $\delta = 1$, which matches the theoretical results of Theorem \ref{thm:2}.

Moreover, we replicate the above experiments under time-dependent noise assumptions. Figure \ref{fig:total_time} displays the convergence results for the two examples mentioned above. We calculate the slope of the lines in Figure \ref{fig:total_time} using data from iterations after $10^5$. The observed trends are analogous to those in the state-dependent case. Notably, when $\gamma = 4$, the error curve appears flat due to limitations in computational accuracy. For this reason, we compute the slope using data from iterations between $10^5$ and $10^6$ for $\gamma = 4$. As depicted in Figure \ref{fig:total_time}, the convergence rate increases with higher values of $\gamma$, and all experiments in this problem converge even faster than the theoretical predictions.

\section{Conclusion}
\label{sec:conclusion}

This paper studies the convergence rate of two-time-scale SA under state- and time-dependent noises. The noises decay with respect to the state residual variables $\Vert \hat{x}_k \Vert^2$ and $\Vert \hat{y}_k \Vert^2$, or time (iteration number) $k$, respectively. We introduce a linear programming function $m(x)$ to establish a general relationship between the parameters $\delta_{ij}$ and the convergence rate of $V_k$. Specifically, in the state-dependent case where $\delta_{ij} \in [0, 1)$, the Lyapunov function $V_k$ decays at a polynomial rate of $\mathcal{O}(k^{-t})$, and a similar result holds in the time-dependent case. Moreover, we derive explicit solutions for the optimal step size when all noise decay parameters are identical. 

Furthermore, it is notable that two-time-scale algorithms can achieve an exponential convergence rate under the condition that the variances of noise are bounded by quadratic forms of the state variables, i.e., $\| \hat{x}_k \|^2$ and $\| \hat{y}_k \|^2$.

We present two numerical examples: SGD with Polyak-Ruppert averaging and stochastic bilevel optimization. The convergence results corroborate our theoretical insights. As $\delta_{ij} \to 1^{-}$ or $\gamma_{i} \to + \infty$, the exponent $t$ in the convergence rate $\mathcal{O} \left( k^{-t} \right)$ increases, potentially reaching $\mathcal{O} \left( k^{-\infty} \right)$. Furthermore, numerical experiments indicate that when $\delta_{ij} \equiv 1$, the exponential convergence rate of $\mathcal{O} \left( e^{-\epsilon k} \right)$ in Theorem \ref{thm:2} can indeed be achieved.

In conclusion, our findings show that under the assumptions outlined in \citet{ilandarideva2023accelerated} and \citet{karandikar2023convergence}, two-time-scale SA can achieve a convergence rate faster than $\mathcal{O} \left( k^{-1} \right)$ and, in certain cases, even attain an exponential convergence rate of $\mathcal{O} \left( e^{-\epsilon k} \right)$.

\section*{Acknowledgments}
Research support from the National Key R\&D Program of China (2022YFA1007900), the National Natural Science Foundation of China (12271013,72342004), Yinhua Education Foundation, and the Fundamental Research Funds for the Central Universities is gratefully acknowledged.
\bibliography{aaai25}

\begin{thebibliography}{32}
\providecommand{\natexlab}[1]{#1}

\bibitem[{Abergel et~al.(2016)Abergel, Anane, Chakraborti, Jedidi, and Toke}]{abergel2016limit}
Abergel, F.; Anane, M.; Chakraborti, A.; Jedidi, A.; and Toke, I.~M. 2016.
\newblock \emph{Limit order books}.
\newblock Cambridge University Press.

\bibitem[{Borkar(1997)}]{borkar1997stochastic}
Borkar, V.~S. 1997.
\newblock Stochastic approximation with two time scales.
\newblock \emph{Systems \& Control Letters}, 29(5): 291--294.

\bibitem[{Borkar(2009)}]{borkar2009stochastic}
Borkar, V.~S. 2009.
\newblock \emph{Stochastic approximation: a dynamical systems viewpoint}, volume~48.
\newblock Springer.

\bibitem[{Bottou, Curtis, and Nocedal(2018)}]{bottou2018optimization}
Bottou, L.; Curtis, F.~E.; and Nocedal, J. 2018.
\newblock Optimization methods for large-scale machine learning.
\newblock \emph{SIAM review}, 60(2): 223--311.

\bibitem[{Chen et~al.(2020)Chen, Lee, Tong, and Zhang}]{chen2020statistical}
Chen, X.; Lee, J.~D.; Tong, X.~T.; and Zhang, Y. 2020.
\newblock Statistical inference for model parameters in stochastic gradient descent.
\newblock \emph{Annals of Statistics}, 48(1): 251--273.

\bibitem[{Colson, Marcotte, and Savard(2007)}]{colson2007overview}
Colson, B.; Marcotte, P.; and Savard, G. 2007.
\newblock An overview of bilevel optimization.
\newblock \emph{Annals of operations research}, 153: 235--256.

\bibitem[{Doan(2022)}]{doan2022nonlinear}
Doan, T.~T. 2022.
\newblock Nonlinear two-time-scale stochastic approximation convergence and finite-time performance.
\newblock \emph{IEEE Transactions on Automatic Control}.

\bibitem[{Doan(2024)}]{doan2024fast}
Doan, T.~T. 2024.
\newblock Fast Nonlinear Two-Time-Scale Stochastic Approximation: Achieving $\mathcal{O} (1/k) $ Finite-Sample Complexity.
\newblock \emph{arXiv preprint arXiv:2401.12764}.

\bibitem[{Ermoliev(1983)}]{ermoliev1983stochastic}
Ermoliev, Y. 1983.
\newblock Stochastic quasigradient methods and their application to system optimization.
\newblock \emph{Stochastics: An International Journal of Probability and Stochastic Processes}, 9(1-2): 1--36.

\bibitem[{Ermoliev(2009)}]{Ermoliev2009}
Ermoliev, Y. 2009.
\newblock \emph{Two-stage stochastic programming: quasigradient methodTwo-Stage Stochastic Programming: Quasigradient Method}, 3955--3959.
\newblock Boston, MA: Springer US.
\newblock ISBN 978-0-387-74759-0.

\bibitem[{Gatheral(2011)}]{gatheral2011volatility}
Gatheral, J. 2011.
\newblock \emph{The volatility surface: A Practitioner's Guide}.
\newblock John Wiley and Sons, Inc.

\bibitem[{Han, Li, and Zhang(2024)}]{han2024finite}
Han, Y.; Li, X.; and Zhang, Z. 2024.
\newblock Finite-Time Decoupled Convergence in Nonlinear Two-Time-Scale Stochastic Approximation.
\newblock \emph{arXiv preprint arXiv:2401.03893}.

\bibitem[{Haykin(2002)}]{haykin2002adaptive}
Haykin, S.~S. 2002.
\newblock \emph{Adaptive filter theory}.
\newblock Pearson Education India.

\bibitem[{Hu, Doshi, and Eun(2024)}]{hu2024central}
Hu, J.; Doshi, V.; and Eun, D.~Y. 2024.
\newblock Central Limit Theorem for Two-Timescale Stochastic Approximation with Markovian Noise: Theory and Applications.
\newblock \emph{arXiv preprint arXiv:2401.09339}.

\bibitem[{Ilandarideva et~al.(2023)Ilandarideva, Juditsky, Lan, and Li}]{ilandarideva2023accelerated}
Ilandarideva, S.; Juditsky, A.; Lan, G.; and Li, T. 2023.
\newblock Accelerated stochastic approximation with state-dependent noise.
\newblock \emph{arXiv preprint arXiv:2307.01497}.

\bibitem[{Kaledin et~al.(2020)Kaledin, Moulines, Naumov, Tadic, and Wai}]{kaledin2020finite}
Kaledin, M.; Moulines, E.; Naumov, A.; Tadic, V.; and Wai, H.-T. 2020.
\newblock Finite time analysis of linear two-timescale stochastic approximation with Markovian noise.
\newblock In \emph{Conference on Learning Theory}, 2144--2203. PMLR.

\bibitem[{Karandikar and Vidyasagar(2023)}]{karandikar2023convergence}
Karandikar, R.~L.; and Vidyasagar, M. 2023.
\newblock Convergence Rates for Stochastic Approximation: Biased Noise with Unbounded Variance, and Applications.
\newblock \emph{arXiv preprint arXiv:2312.02828}.

\bibitem[{Khodadadian et~al.(2022)Khodadadian, Doan, Romberg, and Maguluri}]{khodadadian2022finite}
Khodadadian, S.; Doan, T.~T.; Romberg, J.; and Maguluri, S.~T. 2022.
\newblock Finite sample analysis of two-time-scale natural actor-critic algorithm.
\newblock \emph{IEEE Transactions on Automatic Control}.

\bibitem[{Konda and Tsitsiklis(2004)}]{konda2004convergence}
Konda, V.~R.; and Tsitsiklis, J.~N. 2004.
\newblock Convergence rate of linear two-time-scale stochastic approximation.
\newblock \emph{The Annals of Applied Probability}, 14(2): 796--819.

\bibitem[{Lan(2020)}]{lan2020first}
Lan, G. 2020.
\newblock \emph{First-order and stochastic optimization methods for machine learning}, volume~1.
\newblock Springer.

\bibitem[{Moulines and Bach(2011)}]{moulines2011non}
Moulines, E.; and Bach, F. 2011.
\newblock Non-asymptotic analysis of stochastic approximation algorithms for machine learning.
\newblock \emph{Advances in neural information processing systems}, 24.

\bibitem[{Pham(2010)}]{pham2010new}
Pham, K. 2010.
\newblock New risk-averse control paradigm for stochastic two-time-scale systems and performance robustness.
\newblock \emph{Journal of optimization theory and applications}, 146: 511--537.

\bibitem[{Polyak and Juditsky(1992)}]{polyak1992acceleration}
Polyak, B.~T.; and Juditsky, A.~B. 1992.
\newblock Acceleration of stochastic approximation by averaging.
\newblock \emph{SIAM journal on control and optimization}, 30(4): 838--855.

\bibitem[{Robbins and Monro(1951)}]{robbins1951stochastic}
Robbins, H.; and Monro, S. 1951.
\newblock A stochastic approximation method.
\newblock \emph{The annals of mathematical statistics}, 400--407.

\bibitem[{Ruppert(1988)}]{ruppert1988efficient}
Ruppert, D. 1988.
\newblock Efficient estimations from a slowly convergent Robbins-Monro process.
\newblock Technical report, Cornell University Operations Research and Industrial Engineering.

\bibitem[{Sebbouh, Gower, and Defazio(2021)}]{sebbouh2021almost}
Sebbouh, O.; Gower, R.~M.; and Defazio, A. 2021.
\newblock Almost sure convergence rates for stochastic gradient descent and stochastic heavy ball.
\newblock In \emph{Conference on Learning Theory}, 3935--3971. PMLR.

\bibitem[{Shen and Chen(2022)}]{shen2022single}
Shen, H.; and Chen, T. 2022.
\newblock A single-timescale analysis for stochastic approximation with multiple coupled sequences.
\newblock \emph{Advances in Neural Information Processing Systems}, 35: 17415--17429.

\bibitem[{Sutton et~al.(2009)Sutton, Maei, Precup, Bhatnagar, Silver, Szepesv{\'a}ri, and Wiewiora}]{sutton2009fast}
Sutton, R.~S.; Maei, H.~R.; Precup, D.; Bhatnagar, S.; Silver, D.; Szepesv{\'a}ri, C.; and Wiewiora, E. 2009.
\newblock Fast gradient-descent methods for temporal-difference learning with linear function approximation.
\newblock In \emph{Proceedings of the 26th annual international conference on machine learning}, 993--1000.

\bibitem[{Sutton, Szepesv{\'a}ri, and Maei(2008)}]{sutton2008convergent}
Sutton, R.~S.; Szepesv{\'a}ri, C.; and Maei, H.~R. 2008.
\newblock A convergent O(n) algorithm for off-policy temporal-difference learning with linear function approximation.
\newblock \emph{Advances in neural information processing systems}, 21(21): 1609--1616.

\bibitem[{Wu et~al.(2020)Wu, Zhang, Xu, and Gu}]{wu2020finite}
Wu, Y.~F.; Zhang, W.; Xu, P.; and Gu, Q. 2020.
\newblock A finite-time analysis of two time-scale actor-critic methods.
\newblock \emph{Advances in Neural Information Processing Systems}, 33: 17617--17628.

\bibitem[{Xu, Zou, and Liang(2019)}]{xu2019two}
Xu, T.; Zou, S.; and Liang, Y. 2019.
\newblock Two time-scale off-policy TD learning: Non-asymptotic analysis over Markovian samples.
\newblock \emph{Advances in neural information processing systems}, 32.

\bibitem[{Zeng, Doan, and Romberg(2024)}]{zeng2024two}
Zeng, S.; Doan, T.~T.; and Romberg, J. 2024.
\newblock A two-time-scale stochastic optimization framework with applications in control and reinforcement learning.
\newblock \emph{SIAM Journal on Optimization}, 34(1): 946--976.

\end{thebibliography}

\onecolumn
\setcounter{equation}{0}\renewcommand{\theequation}{S\arabic{equation}}
\setcounter{figure}{0}\renewcommand{\thefigure}{S\arabic{figure}}
\setcounter{table}{0}\renewcommand{\thetable}{S\arabic{table}}
\setcounter{thm}{0}\renewcommand{\thethm}{S\arabic{thm}}
\setcounter{lemma}{0}\renewcommand{\thelemma}{S\arabic{lemma}}
\setcounter{corollary}{0}\renewcommand{\thecorollary}{S\arabic{corollary}}
\setcounter{assumption}{0}\renewcommand{\theassumption}{S\arabic{assumption}}

\clearpage
\appendix
\renewcommand{\thesection}{A.\arabic{section}}
\renewcommand{\thetable}{A.\arabic{table}}
\renewcommand{\thefigure}{A.\arabic{figure}}
\section*{Supplementary Material}
\addcontentsline{toc}{section}{Supplementary Material}

%

\date{}

\thispagestyle{empty}

\setcounter{page}{1}
\pagenumbering{arabic}
\setlength{\baselineskip}{1.5\baselineskip}

\setcounter{equation}{0}
\setcounter{table}{0}
\setcounter{figure}{0}

\section{Introduction}
\label{supp:sec:intro}

This supplementary material provides precise definitions and theorem statements. Mathematical proof can be found in Appendices.

Define the fast-scale as $x_k \in \mathbb{R}^{d_1} $ and the slow-scale $y_k \in \mathbb{R}^{d_2} $ with unknown update functions $ f : \mathbb{R}^{d_1} \times \mathbb{R}^{d_2} \to \mathbb{R}^{d_1} $ and $ g : \mathbb{R}^{d_1} \times \mathbb{R}^{d_2} \to \mathbb{R}^{d_2} $. $f( x_k , y_k ) + \xi_k$ and $ g( x_k ,y_k ) + \psi_k$ are their noisy observations. We consider the following two-time-scale process:
\begin{equation}
\begin{aligned}
     x_{k+1}-x_k = - \alpha_k [ f( x_k , y_k ) + \xi_k ], \\
     y_{k+1}-y_k = - \beta_k [ g( x_k ,y_k ) + \psi_k],
     \label{supp:1}
\end{aligned}
\end{equation}
where $\{\xi_k\}$ and $\{\psi_k\}$ are martingale differences with $\alpha_k$ and $\beta_k$ decreasing.

\section{Problem Setup}
\label{supp:sec:problem_setup}

\begin{assumption}
\label{supp:asm:lambda}
Given $y$ there exists an operator $\lambda$ such that $x=\lambda(y)$ is the unique solution of
\begin{align}
    f(\lambda(y), y)=0.
\end{align}
Suppose $\lambda$ is Lipschitz continuous with respect to constant $L_\lambda$,
\begin{align}
    \Vert \lambda(y_1) - \lambda(y_2) \Vert \le L_\lambda \Vert y_1 - y_2 \Vert.
\end{align}
\end{assumption}

\begin{assumption}
\label{supp:asm:f}
$f$ is Lipschitz continuous with positive constant $L_f$, i.e. $\forall x_1, x_2, y_1$, and $y_2$,
\begin{align}
    \Vert f( x_1 ,y_1) - f( x_2 , y_2) \Vert &\le L_f (\Vert x_1 - x_2 \Vert + \Vert y_1 - y_2 \Vert),
\end{align}
and $f$ is strongly monotone w.r.t $x$ when $y$ is fixed, i.e. there exists a constant $\mu_f > 0$,
\begin{align}
    \langle x_1 - x_2, f(x_1, y) - f(x_2, y) \rangle \ge \mu_f \Vert x_1 - x_2 \Vert^2.
\end{align}
\end{assumption}

\begin{assumption}
\label{supp:asm:g}
The operator $g(\cdot, \cdot)$ is Lipschitz continuous with constant $L_g$, i.e. $\forall x_1, x_2, y_1$, and $y_2$,
\begin{align}
    \Vert g(x_1,y_1) - g(x_2, y_2) \Vert \le L_g ( \Vert x_1 - x_2 \Vert + \Vert y_1 - y_2 \Vert).
\end{align}
Moreover, $g$ is 1-point strongly monotone w.r.t $y^*$, i.e. there exists a constant $\mu_g > 0$ such that
\begin{align}
    \langle y - y^*, g(\lambda(y), y) \rangle &\ge \mu_g \Vert y - y^* \Vert^2.
\end{align}
\end{assumption}

To better illustrate the update process, we introduce two residual variables
\begin{align}
    \hat{x}_k &= x_k - \lambda(y_k),\\
    \hat{y}_k &= y_k - y^*.
\end{align}

\begin{assumption}
\label{supp:asm:xi-psi}
We denote by $\mathcal{Q}_k$ the filtration containing all the history generated by up to the iteration $k$, i.e.,
\begin{align}
    \mathcal{Q}_k = \{ x_0, y_0, \xi_0, \psi_0, \xi_1, \psi_1, \ldots, \xi_k, \psi_k\} ,
    \label{supp:12}
\end{align}
where $\{\xi_k\}_{k \ge 0}$ are independent random variables with zero mean and bounded variances a.s., and so $\{\psi_k\}_{k \ge 0}$ are. Specifically, there exist $\Gamma_{11}$ and $\Gamma_{22}$ such that variances can be controlled as follows almost surely,
\begin{align}
    \mathbb{E}[\Vert \xi_k \Vert ^2 | \mathcal{Q}_{k-1}] &\le \Gamma_{11} \Vert \hat{x}_k \Vert ^ {2 \delta_{11}} + \Gamma_{12} \Vert \hat{y}_k \Vert ^ {2 \delta_{12}} , \\
    \mathbb{E}[\Vert \psi_k \Vert ^2 | \mathcal{Q}_{k-1}] &\le \Gamma_{21} \Vert \hat{x}_k \Vert ^ {2 \delta_{21}} + \Gamma_{22} \Vert \hat{y}_k \Vert ^ {2 \delta_{22}} .
\end{align}
where $\delta_{11}, \delta_{12}, \delta_{21}$, and $\delta_{22}$ are in $[0, 1)$. Meanwhile, denote $\Delta_{ij} = 1 - \delta_{ij}$ for all $i,j \in \{1, 2\}$.
\end{assumption}

\begin{assumption}
\label{supp:asm:xi-psi_2}
$\{\xi_k\}_{k \ge 0}$ and $\{\psi_k\}_{k \ge 0}$ are martingale sequences similarly defined in Assumption \ref{supp:asm:xi-psi} with $\delta_{ij} \equiv 1$, in other words a.s.,
\begin{align}
    \mathbb{E}[\Vert \xi_k \Vert ^2 | \mathcal{Q}_{k-1}] &\le \Gamma_{11} \Vert \hat{x}_k \Vert ^ 2 + \Gamma_{12} \Vert \hat{y}_k \Vert ^ 2, \\
    \mathbb{E}[\Vert \psi_k \Vert ^2 | \mathcal{Q}_{k-1}] &\le \Gamma_{21} \Vert \hat{x}_k \Vert ^ 2 + \Gamma_{22} \Vert \hat{y}_k \Vert ^ 2 .
\end{align}
\end{assumption}

For simplicity, we only consider the time-dependent condition under a certain form of coefficients $\alpha_k$ and $\beta_k$,
\begin{align}
    \alpha_k = \frac{\alpha}{(k + 1 + k_0)^a}, \beta_k = \frac{\beta}{(k + 1 + k_0)^b}.
\end{align}

\begin{assumption}
\label{supp:asm:xi-psi_3}
$\{\xi_k\}_{k \ge 0}$ and $\{\psi_k\}_{k \ge 0}$ are similarly defined in \ref{supp:12}, and for fixed $k_0$ there exist constants $\Gamma_{11}', \Gamma_{22}', \gamma_1, \gamma_2 \ge 0$ such that $\gamma_1 - \gamma_2 \in [-1, \frac{1}{2}) $ and variances can be controlled as follows almost surely,
\begin{align}
    \mathbb{E}[\Vert \xi_k \Vert ^2 | \mathcal{Q}_{k-1}] &\le \Gamma_{11}' (k+1+k_0)^{-\gamma_1}, \\
    \mathbb{E}[\Vert \psi_k \Vert ^2 | \mathcal{Q}_{k-1}] &\le \Gamma_{22}' (k+1+k_0)^{-\gamma_2} .
\end{align}
\end{assumption}

In general, we have
\begin{align}
    \hat{x}_{k+1} &= \hat{x}_k - \alpha_k f(x_k, y_k) + \lambda(y_k) - \lambda(y_{k+1}) - \alpha_k \xi_k, \\
    \hat{y}_{k+1} &= \hat{y}_k - \beta_k g(\lambda(y_k), y_k) + \beta_k ( g(\lambda(y_k), y_k) - g(x_k, y_k) ) - \beta_k \psi_k .
\end{align}

\begin{lemma}
\label{supp:lem:1}
Suppose that Assumptions \ref{supp:asm:lambda}--\ref{supp:asm:xi-psi} hold. Let $x_k$ and $y_k$ be updated by (\ref{supp:1}). Then, for all $k \ge 0$, we have
\begin{align}
        &\mathbb{E} [\Vert \hat{x}_{k+1} \Vert ^2 | \mathcal{Q}_{k-1}] \notag \\
        \le& (1 - \mu_f \alpha_k) \Vert \hat{x}_k \Vert ^2 + (2 L_\lambda L_g + \frac{2}{\mu_g} L_\lambda^2 L_g^2 (L_\lambda + 1)^2 ) \beta_k \Vert \hat{x}_k \Vert ^2 \notag \\
        &+ ( L_{f}^2 \alpha_k^2 + 4 L_\lambda^2  L_g^2 \beta_k^2 + 2 L_f L_\lambda L_g \alpha_k \beta_k + \frac{2}{\mu_g} L_f^2 L_\lambda^2 L_g^2 (L_\lambda + 1)^2 \alpha_k^2 \beta_k) \Vert \hat{x}_k \Vert ^2 \notag \\
        &+ \mu_g \beta_k \Vert \hat{y}_k \Vert ^2 + 4 L_\lambda^2  L_g^2 ( L_{\lambda} + 1)^2 \beta_k^2 \Vert \hat{y}_k \Vert ^2 \notag \\
        &+ \frac{2}{\mu_f \alpha_k} (1 + L_f^2 \alpha_k^2) L_f^2 \beta_k^2 \mathbb{E} [ \Vert \psi_k \Vert ^2 | \mathcal{Q}_{k-1}] + 2L_\lambda^2 \beta_k^2 \mathbb{E}[\Vert \psi_k \Vert^2 | \mathcal{Q}_{k-1}] + 2 \alpha_k^2 \mathbb{E}[\Vert \xi_k \Vert^2 | \mathcal{Q}_{k-1}].
\end{align}
\end{lemma}
\begin{proof}
    We provide a detailed proof in Appendix \ref{supp:proof_of_lem:1}.
\end{proof}

\begin{lemma}
\label{supp:lem:2}
Suppose that Assumptions \ref{supp:asm:lambda}--\ref{supp:asm:xi-psi} hold. Let $x_k$ and $y_k$ be generated by (\ref{supp:1}). Then, for all $k \ge 0$, we have
\begin{align}
    &\mathbb{E} [\Vert \hat{y}_{k+1} \Vert ^2 |\mathcal{Q}_{k-1}] \notag \\
    \le& (1 - \mu_g \beta_k ) \Vert \hat{y}_k \Vert ^2 + \beta_k^2 \mathbb{E} [\Vert \psi_k \Vert ^2 | \mathcal{Q}_{k-1}] \notag \\
    &+ L_g^2 ( 2 L_\lambda^2 + L_\lambda + 3 ) \beta_k^2 \Vert \hat{y}_k \Vert ^2 + \frac{L_g^2 \beta_k}{\mu_g} \Vert \hat{x}_k \Vert ^2 +  L_g^2 ( L_\lambda + 2) \beta_k^2 \Vert \hat{x}_k \Vert ^2 .
\end{align}
\end{lemma}
\begin{proof}
    We provide a detailed proof in Appendix \ref{supp:proof_of_lem:2}.
\end{proof}

We define $c = 4\frac{L_g^2}{\mu_f \mu_g}$. The Lyapunov function is defined as follows and denoted as $V_k$,
\begin{align}
    V_k = V(\hat{x}_k, \hat{y}_k) = c \frac{\beta}{\alpha} \Vert \hat{x}_k \Vert^2 + \Vert \hat{y}_k \Vert^2.
\end{align}

\begin{lemma}
\label{supp:lem:3}
    Suppose that Assumptions \ref{supp:asm:lambda}--\ref{supp:asm:xi-psi} hold. Let $x_k$ and $y_k$ be generated by (\ref{supp:1}). Suppose $\alpha_k, \beta_k$ decreases (not necessarily strictly), and 
    \begin{align}
        \frac{\alpha_k}{\beta_k} \ge \frac{4}{\mu_f} (2 L_\lambda L_g + \frac{2}{\mu_g} L_\lambda^2 L_g^2 (L_\lambda + 1)^2 ) \lor 2 c \lor \frac{\mu_g}{\mu_f} \lor \frac{\alpha_0}{\beta_0} ,  
    \end{align}
    define 
    \begin{align}
        C(\alpha_0, \beta_0) =& L_{f}^2 + 4 L_\lambda^2 L_g^2 \frac{\beta_0^2}{\alpha_0^2} + 2 L_f L_\lambda L_g \frac{\beta_0}{\alpha_0} + \frac{2}{\mu_g} L_f^2 L_\lambda^2 L_g^2 (L_\lambda + 1)^2 \beta_0 + \frac{1}{c} L_g^2 ( L_\lambda + 2) \frac{\beta_0}{\alpha_0} \notag \\
        &+ c 4 L_\lambda^2  L_g^2 ( L_{\lambda} + 1)^2 \frac{\beta_0^3}{\alpha_0^3} + L_g^2 ( 2 L_\lambda^2 + L_\lambda + 3 ) \frac{\beta_0^2}{\alpha_0^2}.
    \end{align}
    Then we have
    \begin{align}
        &\mathbb{E} [ V_{k+1} ] \notag \\
        \le& ( 1 - \frac{1}{2} \mu_g \beta_k ) \mathbb{E} [ V_k ] + C(\alpha_0, \beta_0) \alpha_k^2 \mathbb{E} [ V_k ] \notag \\
        &+ c 2 \alpha_k \beta_k \mathbb{E}[\Vert \xi_k \Vert^2] + \bigg( c ( \frac{2}{\mu_f} L_f^2 \frac{\beta_k^3}{\alpha_k^2} + \frac{2}{\mu_f} L_f^2 L_\lambda^2 \beta_k^3 + 2 L_\lambda^2 \frac{\beta_k^3}{\alpha_k} ) + \beta_k^2 \bigg) \mathbb{E} [\Vert \psi_k \Vert ^2].
    \end{align}
\end{lemma}
\begin{proof}
    We provide a detailed proof in Appendix \ref{supp:proof_of_lem:3}.
\end{proof}

\section{Convergence Under State-Dependent Noise Assumptions}\label{supp:sec:convergence_state}

In Assumption \ref{supp:asm:xi-psi}, we define a linear programming function
\begin{align}
    m(x) =& \frac{(1+\delta_{11}) x - \delta_{11}}{1 - \delta_{11}} \land \frac{x}{1 - \delta_{12}} \land \frac{(2 - \delta_{21}) (1-x) }{1 - \delta_{21}} \land \frac{2(1-x)}{ 1 - \delta_{22}} ,
\end{align}
which is set to be the minimum of four linear functions. $m(x)$ depending on all four parameters $\delta_{ij}$ is a delicate function set to satisfy
\begin{align}
    & -1 + 2a \notag \\
    \ge& -1+a+t+(1-a-t)\delta_{11} \lor -1+a+t-t\delta_{12} \notag \\
    &\lor -3+4a+t+(1-a-t)\delta_{21} \lor -3+4a+t-t\delta_{22},
\end{align}
where $a=\argmax_{x \in (\frac{1}{2}, 1]} m(x)$ and $t = \max_{x \in (\frac{1}{2}, 1]} m(x)$.

\begin{thm}
\label{supp:thm:1}
    Suppose that Assumptions \ref{supp:asm:lambda}--\ref{supp:asm:xi-psi} hold. Let $x_k$ and $y_k$ be generated by (\ref{supp:1}). We assume
    \begin{gather}
        b=1 , a=\argmax_{x \in (\frac{1}{2}, 1]} m(x) , t = \max_{x \in (\frac{1}{2}, 1]} m(x),  \\
        \frac{\alpha}{\beta} \ge \frac{4}{\mu_f} (2 L_\lambda L_g + \frac{2}{\mu_g} L_\lambda^2 L_g^2 (L_\lambda + 1)^2 ) \lor 2 c \lor \frac{\mu_g}{\mu_f} \lor 1 ,  \\
        \beta \ge \frac{2}{\mu_g}(2a+t) .
    \end{gather}
    Let
    \begin{align}
        C_{\alpha\beta} =& L_{f}^2 + 4 L_\lambda^2 L_g^2 \frac{\beta^2}{\alpha^2} + 2 L_f L_\lambda L_g \frac{\beta}{\alpha} + \frac{2}{\mu_g} L_f^2 L_\lambda^2 L_g^2 (L_\lambda + 1)^2 \beta + \frac{1}{c} L_g^2 ( L_\lambda + 2) \frac{\beta}{\alpha} \notag \\
        &+ c 4 L_\lambda^2  L_g^2 ( L_{\lambda} + 1)^2 \frac{\beta^3}{\alpha^3} + L_g^2 ( 2 L_\lambda^2 + L_\lambda + 3 ) \frac{\beta^2}{\alpha^2} \ge C(\alpha_0, \beta_0) .
    \end{align}
    Set 
    \begin{gather}
        k_0 \ge \alpha^{1/a} \lor \left( \frac{\alpha^2}{\beta} \right)^{1/(2a-1)} \lor \frac{1}{2^{1/t} - 1} \lor \frac{1}{2^{1/2a} - 1} \lor (6 C_{\alpha\beta} \alpha^2)^\frac{1}{2a-1} .
    \end{gather}
    Define \begin{align}
        C_1(M) =& C_{\alpha\beta} \alpha^2 2 M ,\\
        C_2(M) =& c 2 \alpha \beta \Gamma_{11} ( \frac{1}{c} \frac{\alpha}{\beta} 2 M )^{\delta_{11}} + c 2 \alpha \beta \Gamma_{12} ( 2 M )^{\delta_{12}} \notag \\
        &+ \frac{\beta^3}{\alpha^2} \bigg( c ( \frac{2}{\mu_f} L_f^2  + \frac{2}{\mu_f} L_f^2 L_\lambda^2 + 2 L_\lambda^2 ) + 1 \bigg) \Gamma_{21} ( \frac{1}{c} \frac{\alpha}{\beta} 2 M )^{\delta_{21}} \notag \\
        &+ \frac{\beta^3}{\alpha^2} \bigg( c ( \frac{2}{\mu_f} L_f^2  + \frac{2}{\mu_f} L_f^2 L_\lambda^2 + 2 L_\lambda^2 ) + 1 \bigg) \Gamma_{22} ( 2 M )^{\delta_{22}} .
    \end{align}
    Note that $\delta_{ij} < 1$, there exists a $M$ s.t.
    \begin{gather}
        M \ge 3 k_0^t V_0 \lor \frac{3}{a} C_2(M) .
    \end{gather}
    Together, we have $\forall k \ge 0$,
    \begin{align}
        \mathbb{E} [ V_k ] \le \frac{M}{(k + k_0)^t} .
    \end{align}
\end{thm}
\begin{proof}
    We provide a detailed proof in Appendix \ref{supp:proof_of_thm:1}.
\end{proof}

\begin{corollary}
\label{supp:cor:1}
    Suppose that Assumptions \ref{supp:asm:lambda}--\ref{supp:asm:xi-psi} hold. Let $x_k$ and $y_k$ be generated by (\ref{supp:1}). Moreover, we assume $\delta_{11} = \delta_{12}$ and $\delta_{21} = \delta_{22}$. Recall that $\Delta_{ij} = 1 - \delta_{ij}$, then under the condition of Theorem \ref{supp:thm:1}, 
    \begin{align}
        a = \frac{\Delta_{11} + \Delta_{22}}{ \Delta_{11} + 2 \Delta_{22}}, t = \frac{1 + \Delta_{22}}{\Delta_{11} + 2 \Delta_{22}}.
    \end{align}
    There exist $\alpha, \beta, k_0$, and $M$, s.t. $\forall k \ge 0$,
    \begin{align}
        \mathbb{E} [ V_k ] \le \frac{M}{(k + k_0)^t} .
    \end{align}
\end{corollary}
\begin{proof}
    We provide a detailed proof in Appendix \ref{supp:proof_of_cor:1}.
\end{proof}

Now we investigate the case in Assumption \ref{supp:asm:xi-psi_2} with constant learning rates $\alpha_k = \alpha$ and $\beta_k = \beta$. We define $c = 4\frac{L_g^2}{\mu_f \mu_g}$. The Lyapunov function is defined as follows and denoted as $V_k$,
\begin{align}
    V_k = V(\hat{x}_k, \hat{y}_k) = c \frac{\beta}{\alpha} \Vert \hat{x}_k \Vert^2 + \Vert \hat{y}_k \Vert^2.
\end{align}

\begin{thm}
\label{supp:thm:2}
    Suppose that Assumptions \ref{supp:asm:lambda}--\ref{supp:asm:g},\ref{supp:asm:xi-psi_2} hold. Let $x_k$ and $y_k$ be generated by (\ref{supp:1}). Let 
    \begin{align}
        C_\beta =& L_{f}^2 + 4 L_\lambda^2 L_g^2 + 2 L_f L_\lambda L_g + \frac{2}{\mu_g} L_f^2 L_\lambda^2 L_g^2 (L_\lambda + 1)^2 \beta + \frac{1}{c} L_g^2 ( L_\lambda + 2) \notag \\
        &+ c 4 L_\lambda^2  L_g^2 ( L_{\lambda} + 1)^2 + L_g^2 ( 2 L_\lambda^2 + L_\lambda + 3 ) = B_1 + B_2 \beta,
    \end{align}
    and we set $\omega = \frac{\alpha}{\beta}$, 
    \begin{gather}
        D_1 (\omega) = - \frac{1}{2} \mu_g + c \frac{2}{\mu_f} L_f^2 \frac{1}{\omega^2} \left( \Gamma_{21} \frac{1}{c} \omega + \Gamma_{22} \right) , \\
        D_2 (\omega) = B_1 \omega^2 + c 2 \omega \left( \Gamma_{11} \frac{1}{c} \omega + \Gamma_{12} \right) + ( c 2 L_\lambda^2 \frac{1}{\omega} + 1 ) \left( \Gamma_{21} \frac{1}{c} \omega + \Gamma_{22} \right) , \\
        D_3 (\omega) = B_2 \omega^2 + c \frac{2}{\mu_f} L_f^2 L_\lambda^2 \left( \Gamma_{21} \frac{1}{c} \omega + \Gamma_{22} \right) .
    \end{gather}
    We assume 
    \begin{gather}
        \omega \ge \frac{4}{\mu_f} (2 L_\lambda L_g + \frac{2}{\mu_g} L_\lambda^2 L_g^2 (L_\lambda + 1)^2 ) \lor 2 c \lor \frac{\mu_g}{\mu_f} \lor 1 , \\
        D_1 (\omega) \le - \frac{1}{4} \mu_g.
    \end{gather}
    Then there exists $\beta > 0$ s.t.
    \begin{align}
        e^{-\epsilon} \triangleq 1 + D_1 (\omega) \beta + D_2 (\omega) \beta^2 + D_3 (\omega) \beta^3 < 1.
    \end{align}
    Together, we have $\forall k \ge 0$,
    \begin{align}
        \mathbb{E} [ V_k ] \le e^{-\epsilon k} V_0 .
    \end{align}
\end{thm}
\begin{proof}
    We provide a detailed proof in Appendix \ref{supp:proof_of_thm:2}.
\end{proof}

\section{Convergence Under Time-Dependent Noise Assumptions}\label{supp:sec:convergence_time}

\begin{thm}
\label{supp:thm:3}
    Suppose that Assumptions \ref{supp:asm:lambda}--\ref{supp:asm:g},\ref{supp:asm:xi-psi_3} hold. Let $x_k$ and $y_k$ be generated by (\ref{supp:1}). We assume 
    \begin{gather}
        b=1 , a=\frac{2-\gamma_1+\gamma_2}{3} \in (\frac{1}{2}, 1] , t = \frac{2+2\gamma_1+\gamma_2}{3},  \\
        \frac{\alpha}{\beta} \ge \frac{4}{\mu_f} (2 L_\lambda L_g + \frac{2}{\mu_g} L_\lambda^2 L_g^2 (L_\lambda + 1)^2 ) \lor 2 c \lor \frac{\mu_g}{\mu_f} \lor 1 , \\
        \beta \ge \frac{2}{\mu_g}(2a+t) .
    \end{gather}
    Let
    \begin{align}
        C_{\alpha\beta} =& L_{f}^2 + 4 L_\lambda^2 L_g^2 \frac{\beta^2}{\alpha^2} + 2 L_f L_\lambda L_g \frac{\beta}{\alpha} + \frac{2}{\mu_g} L_f^2 L_\lambda^2 L_g^2 (L_\lambda + 1)^2 \beta + \frac{1}{c} L_g^2 ( L_\lambda + 2) \frac{\beta}{\alpha} \notag \\
        &+ c 4 L_\lambda^2  L_g^2 ( L_{\lambda} + 1)^2 \frac{\beta^3}{\alpha^3} + L_g^2 ( 2 L_\lambda^2 + L_\lambda + 3 ) \frac{\beta^2}{\alpha^2} \ge C(\alpha_0, \beta_0) .
    \end{align}
    Set 
    \begin{gather}
        k_0 \ge \alpha^{1/a} \lor \left( \frac{\alpha^2}{\beta} \right)^{1/(2a-1)} \lor \frac{1}{2^{1/t} - 1} \lor \frac{1}{2^{1/2a} - 1} \lor (6 C_{\alpha\beta} \alpha^2)^\frac{1}{2a-1} .
    \end{gather}
    Define 
    \begin{align}
        C_1(M) =& C_{\alpha\beta} \alpha^2 2 M ,\\
        C_2' =& c 2 \alpha \beta \Gamma_{11}' + \frac{\beta^3}{\alpha^2} \bigg( c ( \frac{2}{\mu_f} L_f^2  + \frac{2}{\mu_f} L_f^2 L_\lambda^2 + 2 L_\lambda^2 ) + 1 \bigg) \Gamma_{22}' .
    \end{align}
    There exists a $M$ s.t.
    \begin{gather}
        M \ge 3 k_0^t V_0 \lor \frac{3}{a} C_2'.
    \end{gather}
    Together, we have $\forall k \ge 0$,
    \begin{align}
        \mathbb{E} [ V_k ] \le \frac{M}{(k + k_0)^t} .
    \end{align}
\end{thm}
\begin{proof}
    We provide a detailed proof in Appendix \ref{supp:proof_of_thm:3}.
\end{proof}

\newpage
\setcounter{equation}{0} \setcounter{table}{0} \setcounter{figure}{0}
\renewcommand{\theequation}{A.\arabic{equation}}
\renewcommand{\thetable}{A.\arabic{table}}
\renewcommand{\theprop}{A.\arabic{prop}}
\renewcommand{\theprop}{A.\arabic{lemma}}
\setcounter{prop}{0}
\setcounter{lemma}{0}
\appendix
\section{Proof of Theoretical Results}
\label{supp:sec:Appendix A}
In this Supplemental Material, we provide proofs for propositions in the article.

\subsection{Proof of Lemma \ref{supp:lem:1}}
\label{supp:proof_of_lem:1}
\begin{proof}
    Recall that $\hat{x} = x - \lambda(y)$. We have 
    \begin{align}
        \Vert \hat{x}_{k+1} \Vert ^2 =& \Vert \hat{x}_k - \alpha_k f(x_k, y_k) \Vert ^2 + \Vert \lambda(y_k) - \lambda(y_{k+1}) - \alpha_k \xi_k \Vert ^2 \notag \\
        &+ 2( \hat{x}_k - \alpha_k f(x_k , y_k))^T (\lambda(y_k) - \lambda(y_{k+1})) - 2\alpha_k (\hat{x}_k - \alpha_k f(x_k , y_k))^T \xi_k .
    \end{align}
    Note that $f(\lambda(y_k), y_k) = 0$ and $g (\lambda(y^*), y^*)=0$. Then, by Assumptions \ref{supp:asm:lambda}--\ref{supp:asm:xi-psi} and the Cauchy-Schwarz inequality, we give a bound for each of these terms. The first term
    \begin{align}
        &\mathbb{E} [\Vert \hat{x}_k - \alpha_k f(x_k, y_k) \Vert ^2 | \mathcal{Q}_{k-1}] \notag \\
        \le& \Vert \hat{x}_k \Vert^2 + \alpha_k^{2}\Vert f(x_k, y_k)-f(\lambda(y_k),y_k) \Vert ^2 - 2\alpha_k \hat{x}_k^T(f(x_k, y_k)-f(\lambda(y_k),y_k)) \notag \\
        \le& (1 - 2 \mu_f \alpha_k + L_{f}^2 \alpha_k^2 )  \Vert \hat{x}_k \Vert ^2 .
    \end{align}
    The second term
    \begin{align}
        &\mathbb{E} [ \Vert \lambda(y_k) - \lambda(y_{k+1}) - \alpha_k \xi_k \Vert ^2 | \mathcal{Q}_{k-1} ] \notag \\
        \le& 2\mathbb{E} [ \Vert \lambda(y_k) - \lambda(y_{k+1})\Vert ^2 | \mathcal{Q}_{k-1} ] + 2\alpha_k^2\mathbb{E}[\Vert \xi_k \Vert^2 | \mathcal{Q}_{k-1}] \notag \\
        \le& 2 L_{\lambda}^2 \beta_k^2 \mathbb{E} [ \Vert g(x_k, y_k) + \psi_k \Vert^2 | \mathcal{Q}_{k-1}] + 2 \alpha_k^2 \mathbb{E}[\Vert \xi_k \Vert^2 | \mathcal{Q}_{k-1}] \notag \\
        \le& 4 L_{\lambda}^2 \beta_k^2 \mathbb{E} [ \Vert g(x_k, y_k) - g(\lambda(y_k), y_k) \Vert^2 + \Vert g(\lambda(y_k), y_k) - g(\lambda(y^*), y^*) \Vert^2 | \mathcal{Q}_{k-1}] \notag \\
        &+2 L_\lambda^2 \beta_k^2 \mathbb{E}[\Vert \psi_k \Vert^2 | \mathcal{Q}_{k-1}] + 2 \alpha_k^2 \mathbb{E}[\Vert \xi_k \Vert^2 | \mathcal{Q}_{k-1}] \notag \\
        \le& 4 L_\lambda^2  L_g^2 \beta_k^2 \Vert \hat{x}_k \Vert ^2 + 4 L_\lambda^2  L_g^2 ( L_{\lambda} + 1)^2 \beta_k^2 \Vert \hat{y}_k \Vert ^2 \notag \\
        &+ 2L_\lambda^2 \beta_k^2 \mathbb{E}[\Vert \psi_k \Vert^2 | \mathcal{Q}_{k-1}] + 2 \alpha_k^2 \mathbb{E}[\Vert \xi_k \Vert^2 | \mathcal{Q}_{k-1}] .
    \end{align}
    The third term
    \begin{align}
        &\mathbb{E}[2( \hat{x}_k - \alpha_k f(x_k , y_k))^T (\lambda(y_k) - \lambda(y_{k+1})) | \mathcal{Q}_{k-1}] \notag \\
        \le& \mathbb{E}[2 \Vert \hat{x}_k - \alpha_k f(x_k , y_k) + \alpha_k f(\lambda(y_k), y_k) \Vert L_\lambda \Vert y_k - y_{k+1} \Vert | \mathcal{Q}_{k-1}] \notag \\
        \le& \mathbb{E}[2 (1 + L_f \alpha_k) L_\lambda \beta_k \Vert \hat{x}_k \Vert ( L_g \Vert \hat{x}_k \Vert + L_g (L_\lambda+1) \Vert \hat{y}_k \Vert +\Vert \psi_k \Vert ) | \mathcal{Q}_{k-1}] \notag \\
        \le& 2 L_\lambda L_g \beta_k \Vert \hat{x}_k \Vert ^2 + 2 L_f L_\lambda L_g \alpha_k \beta_k \Vert \hat{x}_k \Vert ^2 \notag \\
        &+ \frac{2}{\mu_g} (1 + L_f^2 \alpha_k^2) \beta_k L_\lambda^2 L_g^2 (L_\lambda + 1)^2 \Vert \hat{x}_k \Vert ^2 + \mu_g \beta_k \Vert \hat{y}_k \Vert ^2 \notag \\
        &+ \mu_f \alpha_k \Vert \hat{x}_k \Vert ^2 + \frac{2}{\mu_f \alpha_k} (1 + L_f^2 \alpha_k^2) L_f^2 \beta_k^2 \mathbb{E} [ \Vert \psi_k \Vert ^2 | \mathcal{Q}_{k-1}] .
    \end{align}
    where the last inequality uses the Cauchy-Schwarz inequality for the cross terms. The fourth term
    \begin{align}
        & \mathbb{E}[-2\alpha_k (\hat{x}_k - \alpha_k f(x_k , y_k))^T \xi_k | \mathcal{Q}_{k-1}] = 0.
    \end{align}
    According to the above results, we obtain
    \begin{align}
        &\mathbb{E} [\Vert \hat{x}_{k+1} \Vert ^2 | \mathcal{Q}_{k-1}] \notag \\
        \le& (1 - 2 \mu_f \alpha_k + L_{f}^2 \alpha_k^2 ) \Vert \hat{x}_k \Vert ^2 \notag \\
        &+ 4 L_\lambda^2  L_g^2 \beta_k^2 \Vert \hat{x}_k \Vert ^2 + 4 L_\lambda^2  L_g^2 ( L_{\lambda} + 1)^2 \beta_k^2 \Vert \hat{y}_k \Vert ^2 \notag \\
        &+ 2L_\lambda^2 \beta_k^2 \mathbb{E}[\Vert \psi_k \Vert^2 | \mathcal{Q}_{k-1}] + 2 \alpha_k^2 \mathbb{E}[\Vert \xi_k \Vert^2 | \mathcal{Q}_{k-1}] \notag \\
        &+ 2 L_\lambda L_g \beta_k \Vert \hat{x}_k \Vert ^2 + 2 L_f L_\lambda L_g \alpha_k \beta_k \Vert \hat{x}_k \Vert ^2 \notag \\
        &+ \frac{2}{\mu_g} (1 + L_f^2 \alpha_k^2) \beta_k L_\lambda^2 L_g^2 (L_\lambda + 1)^2 \Vert \hat{x}_k \Vert ^2 + \mu_g \beta_k \Vert \hat{y}_k \Vert ^2 \notag \\
        &+ \mu_f \alpha_k \Vert \hat{x}_k \Vert ^2 + \frac{2}{\mu_f \alpha_k} (1 + L_f^2 \alpha_k^2) L_f^2 \beta_k^2 \mathbb{E} [ \Vert \psi_k \Vert ^2 | \mathcal{Q}_{k-1}] \notag \\
        \le& (1 - \mu_f \alpha_k) \Vert \hat{x}_k \Vert ^2 + (2 L_\lambda L_g + \frac{2}{\mu_g} L_\lambda^2 L_g^2 (L_\lambda + 1)^2 ) \beta_k \Vert \hat{x}_k \Vert ^2 \notag \\
        &+ ( L_{f}^2 \alpha_k^2 + 4 L_\lambda^2  L_g^2 \beta_k^2 + 2 L_f L_\lambda L_g \alpha_k \beta_k + \frac{2}{\mu_g} L_f^2 L_\lambda^2 L_g^2 (L_\lambda + 1)^2 \alpha_k^2 \beta_k) \Vert \hat{x}_k \Vert ^2 \notag \\
        &+ \mu_g \beta_k \Vert \hat{y}_k \Vert ^2 + 4 L_\lambda^2  L_g^2 ( L_{\lambda} + 1)^2 \beta_k^2 \Vert \hat{y}_k \Vert ^2 \notag \\
        &+ \frac{2}{\mu_f \alpha_k} (1 + L_f^2 \alpha_k^2) L_f^2 \beta_k^2 \mathbb{E} [ \Vert \psi_k \Vert ^2 | \mathcal{Q}_{k-1}] + 2L_\lambda^2 \beta_k^2 \mathbb{E}[\Vert \psi_k \Vert^2 | \mathcal{Q}_{k-1}] + 2 \alpha_k^2 \mathbb{E}[\Vert \xi_k \Vert^2 | \mathcal{Q}_{k-1}].
    \end{align}
\end{proof}

\subsection{Proof of Lemma \ref{supp:lem:2}}
\label{supp:proof_of_lem:2}

\begin{proof}
    Recall that $\hat{y} = y - y^*$. We have
    \begin{align}
        \hat{y}_{k+1} &= y_{k+1} - y^* = y_k - y^* - \beta_k g(x_k, y_k) - \beta_k \psi_k \notag \\
        &= \hat{y}_k - \beta_k g(\lambda(y_k), y_k) + \beta_k ( g(\lambda(y_k), y_k) - g(x_k, y_k) ) - \beta_k \psi_k ,
    \end{align}
    which implies that
    \begin{align}
        \Vert \hat{y}_{k+1} \Vert ^2 =& \Vert \hat{y}_k - \beta_k g(\lambda(y_k), y_k) \Vert ^2 + \Vert \beta_k ( g(\lambda(y_k), y_k) - g(x_k, y_k) ) - \beta_k \psi_k \Vert ^2 \notag \\
        &+ 2 \beta_k ( \hat{y}_k - \beta_k g(\lambda(y_k), y_k) )^T ( g(\lambda(y_k), y_k) - g(x_k, y_k) ) \notag \\
        &+ 2 \beta_k ( \hat{y}_k - \beta_k g(\lambda(y_k), y_k) )^T \psi_k .
    \end{align}
    We next analyze each term on the right-hand side. First, using $g(\lambda(y^*), y^*) = 0$, we consider the first term
    \begin{align}
        &\Vert \hat{y}_k - \beta_k g(\lambda(y_k), y_k) \Vert ^2 \notag \\
        =& \Vert \hat{y}_k \Vert ^2 - 2 \beta_k \hat{y}_k^T g(\lambda(y_k), y_k) + \beta_k^2 \Vert g(\lambda(y_k), y_k) \Vert ^2 \notag \\
        \le& \Vert \hat{y}_k \Vert ^2 - 2 \mu_g \beta_k \Vert \hat{y}_k \Vert ^2 + \beta_k^2 \Vert g(\lambda(y_k), y_k) - g(\lambda(y^*), y^*) \Vert ^2 \notag \\
        \le& ( 1 - 2 \mu_g \beta_k ) \Vert \hat{y}_k \Vert ^2 + 2 L_g^2 \beta_k^2 ( \Vert \lambda(y_k) - \lambda(y^*) \Vert ^2 + \Vert \hat{y}_k \Vert ^2 ) \notag \\
        \le& ( 1 - 2 \mu_g \beta_k + 2 L_g^2 (L_\lambda ^2 + 1 ) \beta_k^2 ) \Vert \hat{y}_k \Vert ^2 .
    \end{align}
    Next, taking the conditional expectation of the second term on the right-hand side w.r.t $\mathcal{Q}_{k-1}$ and using the property of martingale difference $\psi_{k}$, we have
    \begin{align}
        &\mathbb{E} [ \Vert \beta_k ( g(\lambda(y_k), y_k) - g(x_k, y_k) ) - \beta_k \psi_k \Vert ^2 | \mathcal{Q}_{k-1} ] \notag \\
        \le& \beta_k^2 \Vert g(\lambda(y_k), y_k) - g(x_k, y_k) \Vert ^2 + \beta_k^2 \mathbb{E} [\Vert \psi_k \Vert ^2 | \mathcal{Q}_{k-1}] \notag \\
        \le& L_g^2 \beta_k^2 \Vert \hat{x}_k \Vert ^2 + \beta_k^2 \mathbb{E} [\Vert \psi_k \Vert ^2 | \mathcal{Q}_{k-1}].
    \end{align}
    We consider the third term on the right-hand side
    \begin{align}
        &2 \beta_k ( \hat{y}_k - \beta_k g(\lambda(y_k), y_k) )^T ( g(\lambda(y_k), y_k) - g(x_k, y_k) ) \notag \\
        \le& 2 \beta_k ( \Vert \hat{y}_k \Vert + \beta_k \Vert g(\lambda(y_k), y_k) \Vert ) L_g \Vert \hat{x}_k \Vert \notag \\
        =& 2 L_g \beta_k \left ( \Vert \hat{y}_k \Vert + \beta_k \Vert g(\lambda(y_k), y_k) - g(\lambda(y^*), y^*) \Vert \right ) \Vert \hat{x}_k \Vert  \notag \\
        \le& 2 L_g \beta_k  ( 1 + (L_\lambda + 1) L_g \beta_k) \Vert \hat{x}_k \Vert \Vert \hat{y}_k \Vert \notag \\
        =& 2 L_g \beta_k \Vert \hat{x}_k \Vert \Vert \hat{y}_k \Vert + 2 L_g^2 (L_\lambda + 1) \beta_k^2 \Vert \hat{x}_k \Vert \Vert \hat{y}_k \Vert \notag \\
        \le& \mu_g \beta_k \Vert \hat{y}_k \Vert ^2 + \frac{L_g^2 \beta_k}{\mu_g} \Vert \hat{x}_k \Vert ^2 + L_g^2 ( L_\lambda + 1) \beta_k^2 ( \Vert \hat{x}_k \Vert ^2 + \Vert \hat{y}_k \Vert ^2).
    \end{align}
    where the last inequality is due to the Cauchy-Schwarz inequality. The fourth term
    \begin{align}
        \mathbb{E} [ 2 \beta_k ( \hat{y}_k - \beta_k g(\lambda(y_k), y_k) )^T \psi_k |\mathcal{Q}_{k-1}] = 0.
    \end{align}
    Finally, taking the conditional expectation w.r.t $\mathcal{Q}_{k-1}$, yields
    \begin{align}
        &\mathbb{E} [\Vert \hat{y}_{k+1} \Vert ^2 |\mathcal{Q}_{k-1}] \notag \\
        \le& ( 1 - 2 \mu_g \beta_k + 2 L_g^2 (L_\lambda ^2 + 1 ) \beta_k^2 ) \Vert \hat{y}_k \Vert ^2 + L_g^2 \beta_k^2 \Vert \hat{x}_k \Vert ^2 + \beta_k^2 \mathbb{E} [\Vert \psi_k \Vert ^2 | \mathcal{Q}_{k-1}]  \notag \\
        &+ \mu_g \beta_k \Vert \hat{y}_k \Vert ^2 + \frac{L_g^2 \beta_k}{\mu_g} \Vert \hat{x}_k \Vert ^2 + L_g^2 ( L_\lambda + 1) \beta_k^2 ( \Vert \hat{x}_k \Vert ^2 + \Vert \hat{y}_k \Vert ^2) \notag \\
        =& (1 - \mu_g \beta_k ) \Vert \hat{y}_k \Vert ^2 + \beta_k^2 \mathbb{E} [\Vert \psi_k \Vert ^2 | \mathcal{Q}_{k-1}] \notag \\
        &+ L_g^2 ( 2 L_\lambda^2 + L_\lambda + 3 ) \beta_k^2 \Vert \hat{y}_k \Vert ^2 + \frac{L_g^2 \beta_k}{\mu_g} \Vert \hat{x}_k \Vert ^2 +  L_g^2 ( L_\lambda + 2) \beta_k^2 \Vert \hat{x}_k \Vert ^2 .
    \end{align}
\end{proof}

\subsection{Proof of Lemma \ref{supp:lem:3}}
\label{supp:proof_of_lem:3}

\begin{proof}
    Since
    \begin{gather}
        \frac{1}{2} c \mu_f \beta_k \ge c (2 L_\lambda L_g + \frac{2}{\mu_g} L_\lambda^2 L_g^2 (L_\lambda + 1)^2 ) \frac{\beta_k^2}{\alpha_k} + \frac{L_g^2}{\mu_g} \beta_k, \\
        \frac{1}{2} \mu_g \beta_k \ge c \mu_g \frac{\beta_k^2}{\alpha_k},
    \end{gather}
    using Lemma \ref{supp:lem:1}-\ref{supp:lem:2},  we have
    \begin{align}
        &\mathbb{E} [ V_{k+1} |\mathcal{Q}_{k-1}] \notag \\
        =& c \frac{\beta_k}{\alpha_k} \mathbb{E} [ \Vert \hat{x}_{k+1} \Vert ^2 | \mathcal{Q}_{k-1}] + \mathbb{E} [ \Vert \hat{y}_{k+1} \Vert ^2 | \mathcal{Q}_{k-1}] \notag \\
        \le& c \frac{\beta_k}{\alpha_k} \bigg( (1 - \mu_f \alpha_k) \Vert \hat{x}_k \Vert ^2 + (2 L_\lambda L_g + \frac{2}{\mu_g} L_\lambda^2 L_g^2 (L_\lambda + 1)^2 ) \beta_k \Vert \hat{x}_k \Vert ^2 \notag \\
        &+ ( L_{f}^2 \alpha_k^2 + 4 L_\lambda^2  L_g^2 \beta_k^2 + 2 L_f L_\lambda L_g \alpha_k \beta_k + \frac{2}{\mu_g} L_f^2 L_\lambda^2 L_g^2 (L_\lambda + 1)^2 \alpha_k^2 \beta_k) \Vert \hat{x}_k \Vert ^2 \notag \\
        &+ \mu_g \beta_k \Vert \hat{y}_k \Vert ^2 + 4 L_\lambda^2  L_g^2 ( L_{\lambda} + 1)^2 \beta_k^2 \Vert \hat{y}_k \Vert ^2 \notag \\
        &+ \frac{2}{\mu_f \alpha_k} (1 + L_f^2 \alpha_k^2) L_f^2 \beta_k^2 \mathbb{E} [ \Vert \psi_k \Vert ^2 | \mathcal{Q}_{k-1}] + 2L_\lambda^2 \beta_k^2 \mathbb{E}[\Vert \psi_k \Vert^2 | \mathcal{Q}_{k-1}] + 2 \alpha_k^2 \mathbb{E}[\Vert \xi_k \Vert^2 | \mathcal{Q}_{k-1}] \bigg) \notag \\
        &+ (1 - \mu_g \beta_k ) \Vert \hat{y}_k \Vert ^2 + \beta_k^2 \mathbb{E} [\Vert \psi_k \Vert ^2 | \mathcal{Q}_{k-1}] \notag \\
        &+ L_g^2 ( 2 L_\lambda^2 + L_\lambda + 3 ) \beta_k^2 \Vert \hat{y}_k \Vert ^2 + \frac{L_g^2 \beta_k}{\mu_g} \Vert \hat{x}_k \Vert ^2 +  L_g^2 ( L_\lambda + 2) \beta_k^2 \Vert \hat{x}_k \Vert ^2 \notag \\
        \le& V_k - \frac{1}{2} c \mu_f \beta_k \Vert \hat{x}_k \Vert ^2 - \frac{1}{2} \mu_g \beta_k \Vert \hat{y}_k \Vert ^2 \notag \\
        &+ \bigg( c( L_{f}^2 \alpha_k \beta_k + 4 L_\lambda^2  L_g^2 \frac{\beta_k^3}{\alpha_k} + 2 L_f L_\lambda L_g \beta_k^2 + \frac{2}{\mu_g} L_f^2 L_\lambda^2 L_g^2 (L_\lambda + 1)^2 \alpha_k \beta_k^2 ) \notag \\
        &\qquad \qquad \qquad \qquad \qquad \qquad  \qquad \qquad \qquad \qquad \qquad \qquad + L_g^2 ( L_\lambda + 2) \beta_k^2 \bigg) \Vert \hat{x}_k \Vert ^2 \notag \\
        &+ \bigg( c 4 L_\lambda^2  L_g^2 ( L_{\lambda} + 1)^2 \frac{\beta_k^3}{\alpha_k} + L_g^2 ( 2 L_\lambda^2 + L_\lambda + 3 ) \beta_k^2 \bigg) \Vert \hat{y}_k \Vert ^2 \notag \\
        &+ c 2 \alpha_k \beta_k \mathbb{E}[\Vert \xi_k \Vert^2 | \mathcal{Q}_{k-1}] + \bigg( c ( \frac{2}{\mu_f} L_f^2 \frac{\beta_k^3}{\alpha_k^2} + \frac{2}{\mu_f} L_f^2 L_\lambda^2 \beta_k^3 + 2 L_\lambda^2 \frac{\beta_k^3}{\alpha_k} ) + \beta_k^2 \bigg) \mathbb{E} [\Vert \psi_k \Vert ^2 | \mathcal{Q}_{k-1}].
    \end{align}
    Since $\frac{\alpha_k}{\beta_k} \ge \frac{\mu_g}{\mu_f} $,
    \begin{align}
        &\mathbb{E} [ V_{k+1} |\mathcal{Q}_{k-1}] \notag \\
        \le& ( 1 - \frac{1}{2} \mu_g \beta_k ) V_k \notag \\
        &+ \alpha_k^2 c \frac{\beta_k}{\alpha_k} \Vert \hat{x}_k \Vert ^2 \bigg( L_{f}^2 + 4 L_\lambda^2  L_g^2 \frac{\beta_k^2}{\alpha_k^2} + 2 L_f L_\lambda L_g \frac{\beta_k}{\alpha_k} + \frac{2}{\mu_g} L_f^2 L_\lambda^2 L_g^2 (L_\lambda + 1)^2 \beta_k \notag \\
        &\qquad \qquad \qquad \qquad \qquad \qquad  \qquad \qquad \qquad \qquad \qquad \qquad + \frac{1}{c} L_g^2 ( L_\lambda + 2) \frac{\beta_k}{\alpha_k} \bigg) \notag \\
        &+ \alpha_k^2 \Vert \hat{y}_k \Vert ^2 \bigg( c 4 L_\lambda^2  L_g^2 ( L_{\lambda} + 1)^2 \frac{\beta_k^3}{\alpha_k^3} + L_g^2 ( 2 L_\lambda^2 + L_\lambda + 3 ) \frac{\beta_k^2}{\alpha_k^2} \bigg) \notag \\
        &+ c 2 \alpha_k \beta_k \mathbb{E}[\Vert \xi_k \Vert^2 | \mathcal{Q}_{k-1}] + \bigg( c ( \frac{2}{\mu_f} L_f^2 \frac{\beta_k^3}{\alpha_k^2} + \frac{2}{\mu_f} L_f^2 L_\lambda^2 \beta_k^3 + 2 L_\lambda^2 \frac{\beta_k^3}{\alpha_k} ) + \beta_k^2 \bigg) \mathbb{E} [\Vert \psi_k \Vert ^2 | \mathcal{Q}_{k-1}].
    \end{align}
    Recall that $\alpha_k, \beta_k$ decreasing and $\frac{\alpha_k}{\beta_k} \ge \frac{\alpha_0}{\beta_0}$,
    \begin{align}
        &\mathbb{E} [ V_{k+1} |\mathcal{Q}_{k-1}] \notag \\
        \le& ( 1 - \frac{1}{2} \mu_g \beta_k ) V_k \notag \\
        &+ \alpha_k^2 V_k \bigg( L_{f}^2 + 4 L_\lambda^2 L_g^2 \frac{\beta_0^2}{\alpha_0^2} + 2 L_f L_\lambda L_g \frac{\beta_0}{\alpha_0} + \frac{2}{\mu_g} L_f^2 L_\lambda^2 L_g^2 (L_\lambda + 1)^2 \beta_0 + \frac{1}{c} L_g^2 ( L_\lambda + 2) \frac{\beta_0}{\alpha_0} \notag \\
        &\qquad \qquad \qquad \qquad \qquad \qquad + c 4 L_\lambda^2  L_g^2 ( L_{\lambda} + 1)^2 \frac{\beta_0^3}{\alpha_0^3} + L_g^2 ( 2 L_\lambda^2 + L_\lambda + 3 ) \frac{\beta_0^2}{\alpha_0^2} \bigg) \notag \\
        &+ c 2 \alpha_k \beta_k \mathbb{E}[\Vert \xi_k \Vert^2 | \mathcal{Q}_{k-1}] + \bigg( c ( \frac{2}{\mu_f} L_f^2 \frac{\beta_k^3}{\alpha_k^2} + \frac{2}{\mu_f} L_f^2 L_\lambda^2 \beta_k^3 + 2 L_\lambda^2 \frac{\beta_k^3}{\alpha_k} ) + \beta_k^2 \bigg) \mathbb{E} [\Vert \psi_k \Vert ^2 | \mathcal{Q}_{k-1}] \notag \\
        =& ( 1 - \frac{1}{2} \mu_g \beta_k ) V_k + C(\alpha_0, \beta_0) \alpha_k^2 V_k \notag \\
        &+ c 2 \alpha_k \beta_k \mathbb{E}[\Vert \xi_k \Vert^2 | \mathcal{Q}_{k-1}] + \bigg( c ( \frac{2}{\mu_f} L_f^2 \frac{\beta_k^3}{\alpha_k^2} + \frac{2}{\mu_f} L_f^2 L_\lambda^2 \beta_k^3 + 2 L_\lambda^2 \frac{\beta_k^3}{\alpha_k} ) + \beta_k^2 \bigg) \mathbb{E} [\Vert \psi_k \Vert ^2 | \mathcal{Q}_{k-1}].
    \end{align}
    Taking the expectation, we have
    \begin{align}
        &\mathbb{E} [ V_{k+1} ] \notag \\
        \le& ( 1 - \frac{1}{2} \mu_g \beta_k ) \mathbb{E} [ V_k ] + C(\alpha_0, \beta_0) \alpha_k^2 \mathbb{E} [ V_k ] \notag \\
        &+ c 2 \alpha_k \beta_k \mathbb{E}[\Vert \xi_k \Vert^2] + \bigg( c ( \frac{2}{\mu_f} L_f^2 \frac{\beta_k^3}{\alpha_k^2} + \frac{2}{\mu_f} L_f^2 L_\lambda^2 \beta_k^3 + 2 L_\lambda^2 \frac{\beta_k^3}{\alpha_k} ) + \beta_k^2 \bigg) \mathbb{E} [\Vert \psi_k \Vert ^2].
    \end{align}
\end{proof}

\subsection{Proof of Theorem \ref{supp:thm:1}}
\label{supp:proof_of_thm:1}

\begin{proof}
    Using Lyapunov inequality, we have for $\delta_{ij}$ in $[0,1)$,
    \begin{align}
        \mathbb{E}[\Vert \xi_k \Vert^2] \le \Gamma_{11} \big( \mathbb{E} [ \Vert \hat{x}_k \Vert ^2] \big) ^ {\delta_{11}} + \Gamma_{12} \big( \mathbb{E} [ \Vert \hat{y}_k \Vert ^2] \big) ^ {\delta_{12}}, \\
        \mathbb{E}[\Vert \psi_k \Vert^2] \le \Gamma_{21} \big( \mathbb{E} [ \Vert \hat{x}_k \Vert ^2] \big) ^ {\delta_{21}} + \Gamma_{22} \big( \mathbb{E} [ \Vert \hat{y}_k \Vert ^2] \big) ^ {\delta_{22}}.
    \end{align}
    This yields
    \begin{align}
        \mathbb{E}[\Vert \xi_k \Vert^2] \le \Gamma_{11} \big( \frac{1}{c} \frac{\alpha_k}{\beta_k} \mathbb{E} [ V_k ] \big) ^ {\delta_{11}} + \Gamma_{12} \big( \mathbb{E} [ V_k ] \big) ^ {\delta_{12}}, \\
        \mathbb{E}[\Vert \psi_k \Vert^2] \le \Gamma_{21} \big( \frac{1}{c} \frac{\alpha_k}{\beta_k} \mathbb{E} [ V_k ] \big) ^ {\delta_{21}} + \Gamma_{22} \big( \mathbb{E} [ V_k ] \big) ^ {\delta_{22}}.
    \end{align}
    Combining with Lemma \ref{supp:lem:3}, we get
    \begin{align}
        &\mathbb{E} [ V_{k+1} ] \notag \\
        \le& ( 1 - \frac{1}{2} \mu_g \beta_k ) \mathbb{E} [ V_k ] + C_{\alpha\beta} \alpha_k^2 \mathbb{E} [ V_k ] \notag \\
        &+ c 2 \alpha_k \beta_k \left( \Gamma_{11} \big( \frac{1}{c} \frac{\alpha_k}{\beta_k} \mathbb{E} [ V_k ] \big) ^ {\delta_{11}} + \Gamma_{12} \big( \mathbb{E} [ V_k ] \big) ^ {\delta_{12}} \right) \notag \\
        &+ \bigg( c ( \frac{2}{\mu_f} L_f^2 \frac{\beta_k^3}{\alpha_k^2} + \frac{2}{\mu_f} L_f^2 L_\lambda^2 \beta_k^3 + 2 L_\lambda^2 \frac{\beta_k^3}{\alpha_k} ) + \beta_k^2 \bigg) \left( \Gamma_{21} \big( \frac{1}{c} \frac{\alpha_k}{\beta_k} \mathbb{E} [ V_k ] \big) ^ {\delta_{21}} + \Gamma_{22} \big( \mathbb{E} [ V_k ] \big) ^ {\delta_{22}} \right) \notag \\
        \le& ( 1 - \frac{1}{2} \mu_g \beta_k ) \mathbb{E} [ V_k ] + C_{\alpha\beta} \alpha_k^2 \mathbb{E} [ V_k ] \notag \\
        &+ c 2 \alpha_k \beta_k \left( \Gamma_{11} \big( \frac{1}{c} \frac{\alpha_k}{\beta_k} \mathbb{E} [ V_k ] \big) ^ {\delta_{11}} + \Gamma_{12} \big( \mathbb{E} [ V_k ] \big) ^ {\delta_{12}} \right) \notag \\
        &+ \frac{\beta_k^3}{\alpha_k^2} \bigg( c ( \frac{2}{\mu_f} L_f^2  + \frac{2}{\mu_f} L_f^2 L_\lambda^2 + 2 L_\lambda^2 ) + 1 \bigg) \left( \Gamma_{21} \big( \frac{1}{c} \frac{\alpha_k}{\beta_k} \mathbb{E} [ V_k ] \big) ^ {\delta_{21}} + \Gamma_{22} \big( \mathbb{E} [ V_k ] \big) ^ {\delta_{22}} \right),
    \end{align}
    where $\alpha_k \le 1, \alpha_k / \beta_k \ge 1$, and $\beta_k^2 \le \beta_k^3 / \alpha_k^2$ is due to $\left( \alpha^2 / \beta \right)^{1/(2a-1)}  \le k_0 \le k+1+k_0$. Now we prove it by induction, for $k = 0$, it holds because \begin{align}
        M \ge 3 k_0^t V_0  \ge k_0^t V_0.
    \end{align}
    Using Bernoulli's inequality 
    \begin{gather}
        1 - \frac{1}{2} \mu_g \beta_k \le 1 - \frac{2a+t}{k+1+k_0} \le (1 - \frac{1}{k+1+k_0})^{2a+t}, \\
        \frac{1}{(k + k_0)^t} \le \frac{2}{(k + 1 + k_0)^t} ,
    \end{gather}
    then for $k+1$, from the condition we get
    \begin{align}
        &( k + 1 + k_0 )^{2a+t} \mathbb{E} [ V_{k+1} ] \notag \\
        \le& ( k + k_0 )^{2a+t} \mathbb{E} [ V_k ] + C_{\alpha\beta} \alpha^2 (k+1+k_0)^{t} \frac{2M}{(k + 1 + k_0)^t} \notag \\
        &+ c 2 \alpha \beta (k+1+k_0)^{-1+a+t} \notag \\
        &\qquad \cdot \left( \Gamma_{11} \big( \frac{1}{c} \frac{\alpha}{\beta} ( k+1+k_0 )^{1-a} 2M (k + 1 + k_0)^{-t} \big) ^ {\delta_{11}} + \Gamma_{12} \big( 2M (k + 1 + k_0)^{-t} \big) ^ {\delta_{12}} \right) \notag \\
        &+ \frac{\beta^3}{\alpha^2} ( k+1+k_0 )^{-3+4a+t} \bigg( c ( \frac{2}{\mu_f} L_f^2  + \frac{2}{\mu_f} L_f^2 L_\lambda^2 + 2 L_\lambda^2 ) + 1 \bigg) \notag \\
        & \qquad \cdot \left( \Gamma_{21} \big( \frac{1}{c} \frac{\alpha}{\beta} ( k+1+k_0 )^{1-a} 2M (k + 1 + k_0)^{-t} \big) ^ {\delta_{21}} + \Gamma_{22} \big( 2M (k + 1 + k_0)^{-t} \big) ^ {\delta_{22}} \right) \notag \\
        =& ( k + k_0 )^{2a+t} \mathbb{E} [ V_k ] + C_1(M) \notag \\
        &+ c 2 \alpha \beta \Gamma_{11} ( \frac{1}{c} \frac{\alpha}{\beta} 2 M )^{\delta_{11}} (k+1+k_0)^{-1+a+t+(1-a-t)\delta_{11}} \notag \\
        &+ c 2 \alpha \beta \Gamma_{12} ( 2 M )^{\delta_{12}} (k+1+k_0)^{-1+a+t-t\delta_{12}} \notag \\
        &+ \frac{\beta^3}{\alpha^2} \bigg( c ( \frac{2}{\mu_f} L_f^2  + \frac{2}{\mu_f} L_f^2 L_\lambda^2 + 2 L_\lambda^2 ) + 1 \bigg) \Gamma_{21} ( \frac{1}{c} \frac{\alpha}{\beta} 2 M )^{\delta_{21}} (k+1+k_0)^{-3+4a+t+(1-a-t)\delta_{21}} \notag \\
        &+ \frac{\beta^3}{\alpha^2} \bigg( c ( \frac{2}{\mu_f} L_f^2  + \frac{2}{\mu_f} L_f^2 L_\lambda^2 + 2 L_\lambda^2 ) + 1 \bigg) \Gamma_{22} ( 2 M )^{\delta_{22}} (k+1+k_0)^{-3+4a+t-t\delta_{22}} .
    \end{align}
    And using the property of $m(x)$,
    \begin{align}
        & -1 + 2a \notag \\
        \ge& -1+a+t+(1-a-t)\delta_{11} \lor -1+a+t-t\delta_{12} \notag \\
        &\lor -3+4a+t+(1-a-t)\delta_{21} \lor -3+4a+t-t\delta_{22},
    \end{align}
    we have
    \begin{align}
        ( k + 1 + k_0 )^{2a+t} \mathbb{E} [ V_{k+1} ] \le ( k + k_0 )^{2a+t} \mathbb{E} [ V_k ] + C_1(M) + C_2(M) (k+1+k_0)^{-1+2a}.
    \end{align}
    Because
    \begin{gather}
        \sum_{i=0}^{k} (i+1+k_0)^{-1+2a} \le \int_0^{k+2+k_0} i^{-1+2a} \text{d} i = \frac{1}{2a} (k+2+k_0)^{2a}, \\
        (k+2+k_0)^{2a} \le 2 (k+1+k_0)^{2a},
    \end{gather}
    do telescope sum
    \begin{align}
        ( k + 1 + k_0 )^{2a+t} \mathbb{E} [ V_{k+1} ] \le& k_0^{2a+t} V_0 + C_1(M) (k+1) + C_2(M) \frac{1}{2a} (k+2+k_0)^{2a} \notag \\
        \le& k_0^{2a+t} V_0 + C_1(M) (k+1) + C_2(M) \frac{1}{2a} 2 (k+1+k_0)^{2a}.
    \end{align}
    This tells
    \begin{align}
        \mathbb{E} [ V_{k+1} ] \le& \frac{k_0^{2a+t}}{( k + 1 + k_0 )^{2a+t}} V_0 + C_1(M) \frac{1}{( k + 1 + k_0 )^{-1+2a+t}} + \frac{C_2(M)}{a} \frac{1}{( k + 1 + k_0 )^{t}} \notag \\
        \le& \frac{1}{3} \frac{M}{(k + 1 + k_0)^t} + \frac{1}{3} \frac{M}{(k + 1 + k_0)^t} + \frac{1}{3} \frac{M}{(k + 1 + k_0)^t} \notag \\
        =& \frac{M}{(k + 1 + k_0)^t}.
    \end{align}
    Finally, we come to a conclusion.
\end{proof}

\subsection{Proof of Corollary \ref{supp:cor:1}}
\label{supp:proof_of_cor:1}

\begin{proof}
    Recall that
    \begin{align}
        m(x) =& \frac{(1+\delta_{11}) x - \delta_{11}}{1 - \delta_{11}} \land \frac{x}{1 - \delta_{12}} \land \frac{(2 - \delta_{21}) (1-x) }{1 - \delta_{21}} \land \frac{2(1-x)}{ 1 - \delta_{22}} ,
    \end{align}
    where $\delta_{11} = \delta_{12}, \delta_{21} = \delta_{22}$. Since $\forall x \in (\frac{1}{2} , 1]$,
    \begin{align}
        \frac{(1+\delta_{11}) x - \delta_{11}}{1 - \delta_{11}} \le \frac{x}{1 - \delta_{11}} , \frac{(2 - \delta_{22}) (1-x) }{1 - \delta_{22}} \le \frac{2(1-x)}{ 1 - \delta_{22}},
    \end{align}
    we have
    \begin{gather}
        a=\argmax_{x \in (\frac{1}{2}, 1]} m(x) = \frac{2 - \delta_{11} - \delta_{22}}{ 3 - \delta_{11} - 2 \delta_{22}} = \frac{\Delta_{11} + \Delta_{22}}{ \Delta_{11} + 2 \Delta_{22}}, \\
        t = \max_{x \in (\frac{1}{2}, 1]} m(x) = \frac{(1+\delta_{11}) a - \delta_{11}}{1 - \delta_{11}} = \frac{(2 - \delta_{22}) (1-a) }{1 - \delta_{22}} = \frac{1 + \Delta_{22}}{\Delta_{11} + 2 \Delta_{22}} .
    \end{gather}
    Then according to Theorem \ref{supp:thm:1}, $V_k$ has a convergence rate of $\mathcal{O} \left( k^{-t} \right)$.
\end{proof}

\subsection{Proof of Theorem \ref{supp:thm:2}}
\label{supp:proof_of_thm:2}

\begin{proof}
    Since
    \begin{align}
        \mathbb{E}[\Vert \xi_k \Vert ^2 ] &\le \Gamma_{11} \mathbb{E} [ \Vert \hat{x}_k \Vert ^2] + \Gamma_{12} \mathbb{E} [ \Vert \hat{y}_k \Vert ^2] \le \Gamma_{11} \frac{1}{c} \frac{\alpha}{\beta} \mathbb{E} [ V_k ] + \Gamma_{12} \mathbb{E} [ V_k ] , \\
        \mathbb{E}[\Vert \psi_k \Vert ^2 ] &\le \Gamma_{21} \mathbb{E} [ \Vert \hat{x}_k \Vert ^2] + \Gamma_{22} \mathbb{E} [ \Vert \hat{y}_k \Vert ^2] \le \Gamma_{21} \frac{1}{c} \frac{\alpha}{\beta} \mathbb{E} [ V_k ] + \Gamma_{22} \mathbb{E} [ V_k ],
    \end{align}
    Lemma \ref{supp:lem:3} yields
    \begin{align}
        &\mathbb{E} [ V_{k+1} ] \notag \\
        \le& ( 1 - \frac{1}{2} \mu_g \beta ) \mathbb{E} [ V_k ] + (B_1 + B_2 \beta) \alpha^2 \mathbb{E} [ V_k ] + c 2 \alpha \beta \left( \Gamma_{11} \frac{1}{c} \frac{\alpha}{\beta} \mathbb{E} [ V_k ] + \Gamma_{12} \mathbb{E} [ V_k ] \right) \notag \\
        &+ \bigg( c ( \frac{2}{\mu_f} L_f^2 \frac{\beta^3}{\alpha^2} + \frac{2}{\mu_f} L_f^2 L_\lambda^2 \beta^3 + 2 L_\lambda^2 \frac{\beta^3}{\alpha} ) + \beta^2 \bigg) \left( \Gamma_{21} \frac{1}{c} \frac{\alpha}{\beta} \mathbb{E} [ V_k ] + \Gamma_{22} \mathbb{E} [ V_k ] \right) \notag \\
        =& ( 1 - \frac{1}{2} \mu_g \beta ) \mathbb{E} [ V_k ] + (B_1 + B_2 \beta) \omega^2 \beta^2 \mathbb{E} [ V_k ] + c 2 \omega \beta^2 \left( \Gamma_{11} \frac{1}{c} \omega + \Gamma_{12} \right)  \mathbb{E} [ V_k ] \notag \\
        &+ \bigg( c ( \frac{2}{\mu_f} L_f^2 \frac{1}{\omega^2} \beta + \frac{2}{\mu_f} L_f^2 L_\lambda^2 \beta^3 + 2 L_\lambda^2 \frac{1}{\omega} \beta^2 ) + \beta^2 \bigg) \left( \Gamma_{21} \frac{1}{c} \omega + \Gamma_{22} \right)  \mathbb{E} [ V_k ] \notag \\
        =& ( 1 + D_1 (\omega) \beta + D_2 (\omega) \beta^2 + D_3 (\omega) \beta^3 ) \mathbb{E} [ V_k ] \notag \\
        =& e^{-\epsilon} \mathbb{E} [ V_k ].
    \end{align}
    This gives us, $\forall k \ge 0$
    \begin{align}
        \mathbb{E} [ V_k ] \le e^{-\epsilon k} V_0 .
    \end{align}
\end{proof}

\subsection{Proof of Theorem \ref{supp:thm:3}}
\label{supp:proof_of_thm:3}

\begin{proof}
    Assumption \ref{supp:asm:xi-psi_3} combining with Lemma \ref{supp:lem:3} yields
    \begin{align}
        &\mathbb{E} [ V_{k+1} ] \notag \\
        \le& ( 1 - \frac{1}{2} \mu_g \beta_k ) \mathbb{E} [ V_k ] + C_{\alpha\beta} \alpha_k^2 \mathbb{E} [ V_k ] \notag \\
        &+ c 2 \alpha_k \beta_k \Gamma_{11}' (k+1+k_0)^{-\gamma_1} \notag \\
        &+ \bigg( c ( \frac{2}{\mu_f} L_f^2 \frac{\beta_k^3}{\alpha_k^2} + \frac{2}{\mu_f} L_f^2 L_\lambda^2 \beta_k^3 + 2 L_\lambda^2 \frac{\beta_k^3}{\alpha_k} ) + \beta_k^2 \bigg) \Gamma_{22}' (k+1+k_0)^{-\gamma_2} \\
        \le& ( 1 - \frac{1}{2} \mu_g \beta_k ) \mathbb{E} [ V_k ] + C_{\alpha\beta} \alpha_k^2 \mathbb{E} [ V_k ] \notag \\
        &+ c 2 \alpha_k \beta_k \Gamma_{11}' (k+1+k_0)^{-\gamma_1} \notag \\
        &+ \frac{\beta_k^3}{\alpha_k^2} \bigg( c ( \frac{2}{\mu_f} L_f^2  + \frac{2}{\mu_f} L_f^2 L_\lambda^2 + 2 L_\lambda^2 ) + 1 \bigg) \Gamma_{22}' (k+1+k_0)^{-\gamma_2},
    \end{align}
    where $\alpha_k \le 1, \alpha_k / \beta_k \ge 1$, and $\beta_k^2 \le \beta_k^3 / \alpha_k^2$ is due to $\left( \alpha^2 / \beta \right)^{1/(2a-1)}  \le k_0 \le k+1+k_0$. Now we prove it by induction, for $k = 0$, it holds because \begin{align}
        M \ge 3 k_0^t V_0  \ge k_0^t V_0.
    \end{align}
    Using Bernoulli's inequality 
    \begin{gather}
        1 - \frac{1}{2} \mu_g \beta_k \le 1 - \frac{2a+t}{k+1+k_0} \le (1 - \frac{1}{k+1+k_0})^{2a+t}, \\
        \frac{1}{(k + k_0)^t} \le \frac{2}{(k + 1 + k_0)^t} ,
    \end{gather}
    then for $k+1$, from the condition we get
    \begin{align}
        &( k + 1 + k_0 )^{2a+t} \mathbb{E} [ V_{k+1} ] \notag \\
        \le& ( k + k_0 )^{2a+t} \mathbb{E} [ V_k ] + C_{\alpha\beta} \alpha^2 (k+1+k_0)^{t} \frac{2M}{(k + 1 + k_0)^t} \notag \\
        &+ c 2 \alpha \beta (k+1+k_0)^{-1+a+t} \Gamma_{11}' (k+1+k_0)^{-\gamma_1} \notag \\
        &+ \frac{\beta^3}{\alpha^2} ( k+1+k_0 )^{-3+4a+t} \bigg( c ( \frac{2}{\mu_f} L_f^2  + \frac{2}{\mu_f} L_f^2 L_\lambda^2 + 2 L_\lambda^2 ) + 1 \bigg) \Gamma_{22}' (k+1+k_0)^{-\gamma_2}  \notag \\
        =& ( k + k_0 )^{2a+t} \mathbb{E} [ V_k ] + C_1(M) \notag \\
        &+ c 2 \alpha \beta \Gamma_{11}' (k+1+k_0)^{-1+a+t-\gamma_1} \notag \\
        &+ \frac{\beta^3}{\alpha^2} \bigg( c ( \frac{2}{\mu_f} L_f^2  + \frac{2}{\mu_f} L_f^2 L_\lambda^2 + 2 L_\lambda^2 ) + 1 \bigg) \Gamma_{22}' ( k+1+k_0 )^{-3+4a+t-\gamma_2}.
    \end{align}
    By definition,
    \begin{align}
        -1 + 2a = -1+a+t-\gamma_1 = -3+4a+t-\gamma_2,
    \end{align}
    we have
    \begin{align}
        ( k + 1 + k_0 )^{2a+t} \mathbb{E} [ V_{k+1} ] \le ( k + k_0 )^{2a+t} \mathbb{E} [ V_k ] + C_1(M) + C_2' (k+1+k_0)^{-1+2a}.
    \end{align}
    Because
    \begin{gather}
        \sum_{i=0}^{k} (i+1+k_0)^{-1+2a} \le \int_0^{k+2+k_0} i^{-1+2a} \text{d} i = \frac{1}{2a} (k+2+k_0)^{2a}, \\
        (k+2+k_0)^{2a} \le 2 (k+1+k_0)^{2a},
    \end{gather}
    do telescope sum
    \begin{align}
        ( k + 1 + k_0 )^{2a+t} \mathbb{E} [ V_{k+1} ] \le& k_0^{2a+t} V_0 + C_1(M) (k+1) + C_2' \frac{1}{2a} (k+2+k_0)^{2a} \notag \\
        \le& k_0^{2a+t} V_0 + C_1(M) (k+1) + C_2' \frac{1}{2a} 2 (k+1+k_0)^{2a}.
    \end{align}
    This tells
    \begin{align}
        \mathbb{E} [ V_{k+1} ] \le& \frac{k_0^{2a+t}}{( k + 1 + k_0 )^{2a+t}} V_0 + C_1(M) \frac{1}{( k + 1 + k_0 )^{-1+2a+t}} + \frac{C_2'}{a} \frac{1}{( k + 1 + k_0 )^{t}} \notag \\
        \le& \frac{1}{3} \frac{M}{(k + 1 + k_0)^t} + \frac{1}{3} \frac{M}{(k + 1 + k_0)^t} + \frac{1}{3} \frac{M}{(k + 1 + k_0)^t} \notag \\
        =& \frac{M}{(k + 1 + k_0)^t}.
    \end{align}
    Finally, we come to a conclusion.
\end{proof}

\section{Figures of Numerical Experiments}
\label{supp:sec:Appendix B}

\begin{figure*}[th!]
    \centering
    {
        \includegraphics[width = 0.45\linewidth]{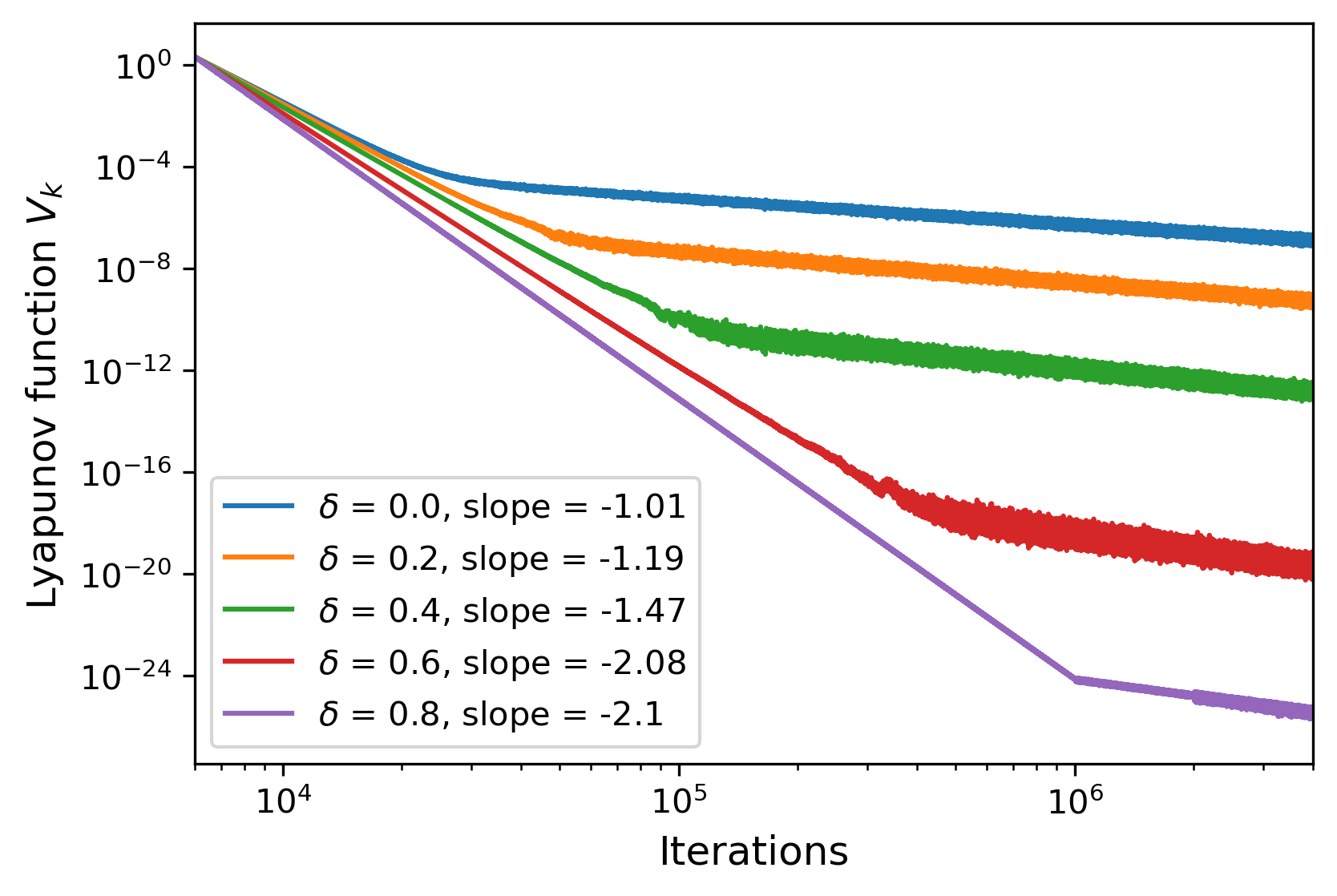}
    }
    {
        \includegraphics[width = 0.45\linewidth]{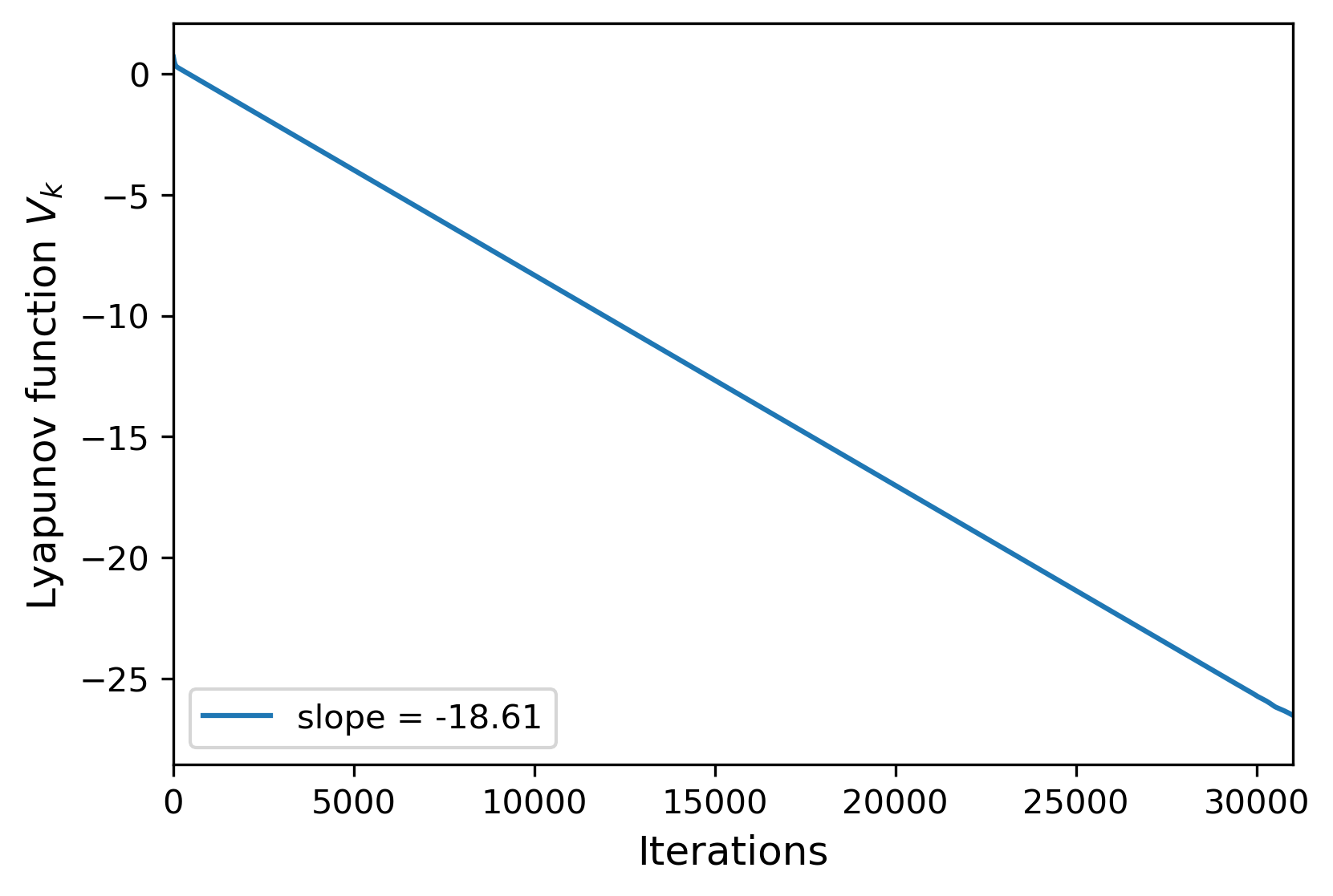}
    }
\caption{The convergence results of SGD with Polyak-Ruppert averaging. The figure on the left is a log-log plot in case $\delta=0.0, 0.2, 0.4, 0.6, 0.8$; the figure on the right is a logarithmic plot in case $\delta=1$. All R-squares of the fitted slopes do not exceed 1e-6.}
\label{supp:fig:total_SGD_PR_averaging}
\end{figure*}

\begin{figure*}[th!]
    \centering
    {
        \includegraphics[width = 0.45\linewidth]{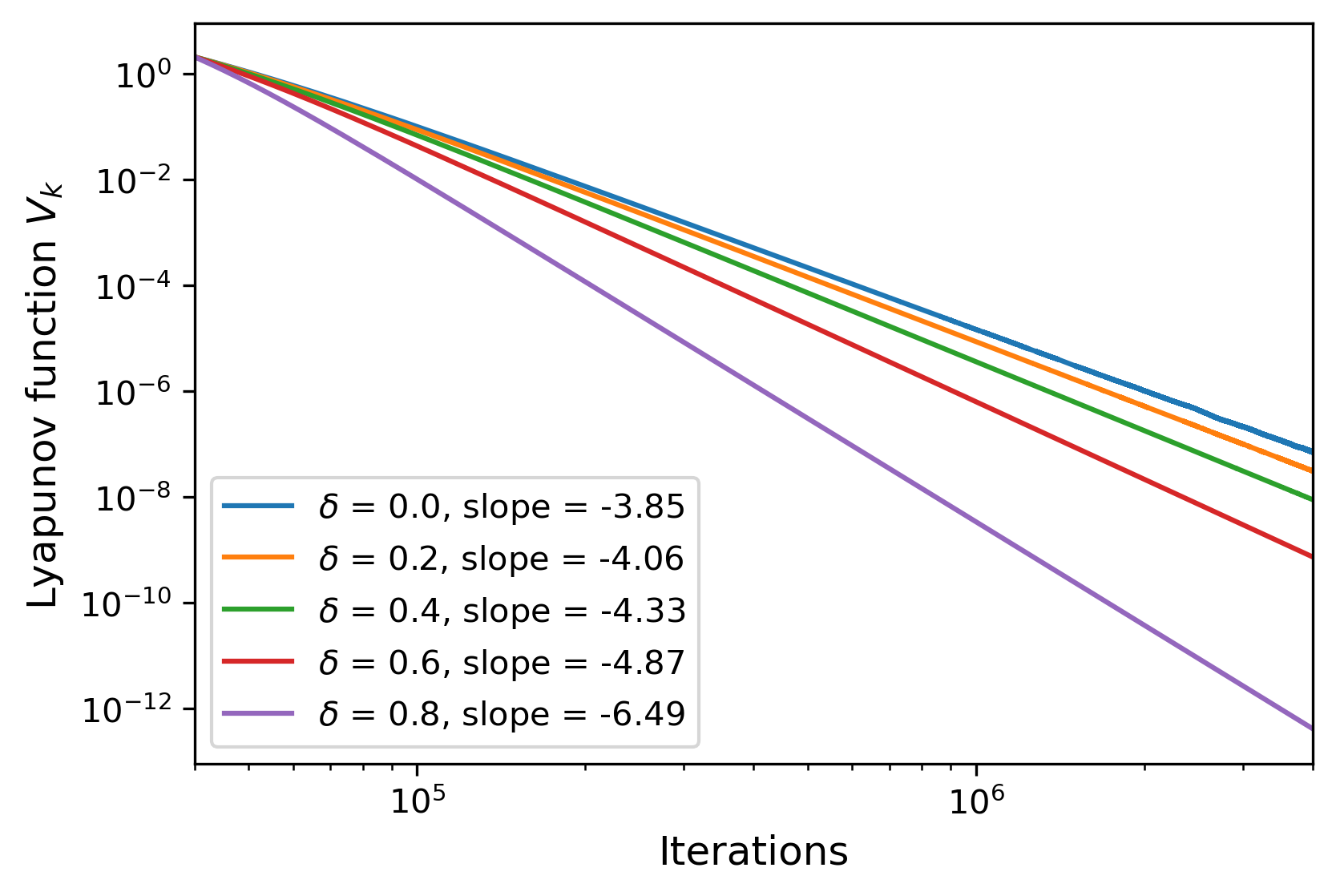}
    }
    {
        \includegraphics[width = 0.45\linewidth]{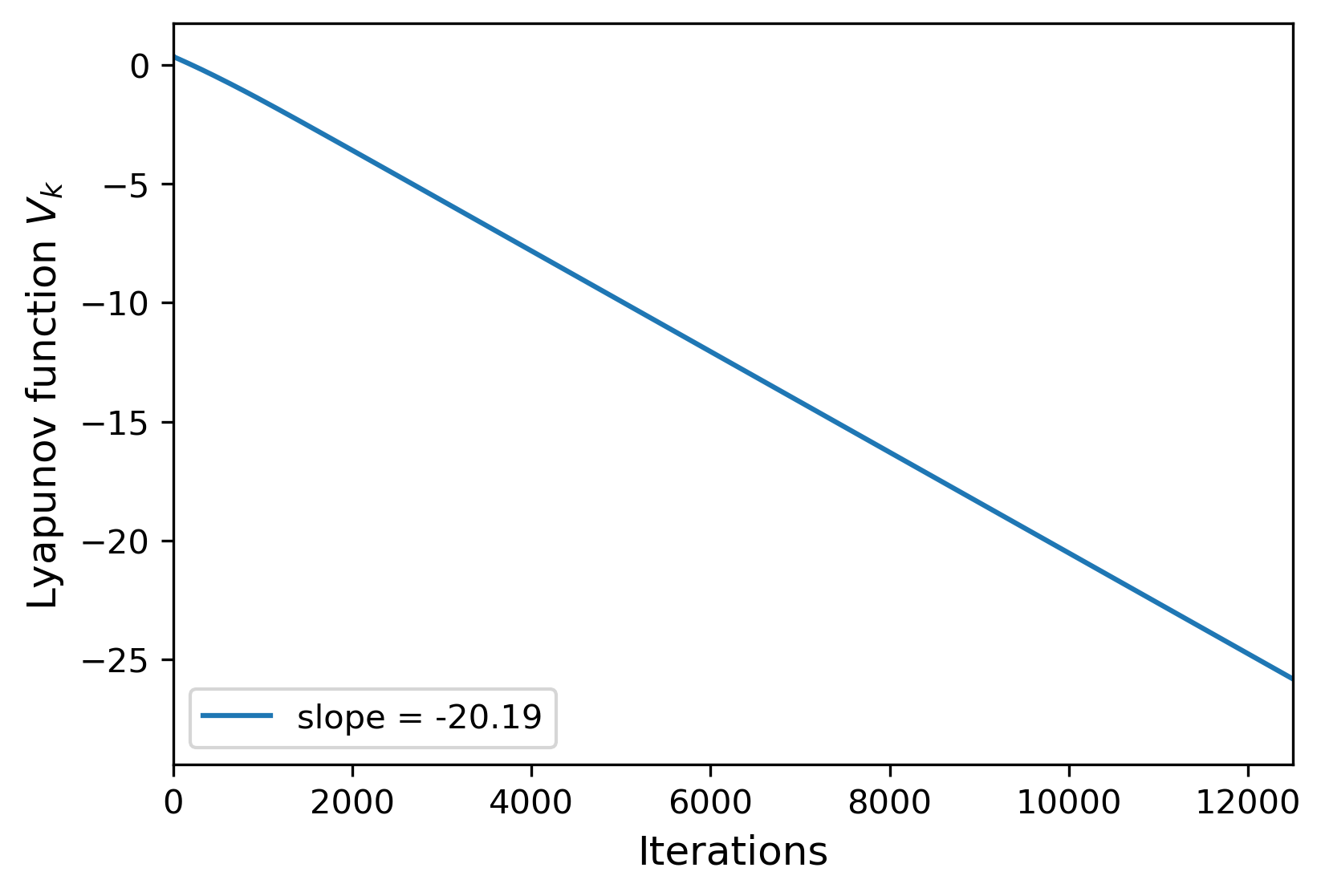}
    }
\caption{The convergence results of stochastic bilevel optimization. The figure on the left is a log-log plot in case $\delta=0.0, 0.2, 0.4, 0.6, 0.8$; the figure on the right is a logarithmic plot in case $\delta=1$. All R-squares of the fitted slopes do not exceed 0.01.}
\label{supp:fig:total_SBO}
\end{figure*}

\begin{figure*}[th!]
    \centering
    {
        \includegraphics[width = 0.45\linewidth]{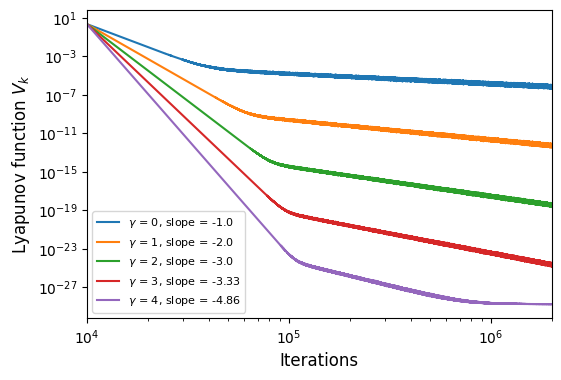}
    }
    {
        \includegraphics[width = 0.45\linewidth]{Figures/time_SBO.png}
    }
\caption{The convergence results of SGD with Polyak-Ruppert averaging (left) and stochastic bilevel optimization (right) under time-dependent noise assumption. The figure is a log-log plot in case $\gamma=0, 1, 2, 3, 4$. All R-squares of the fitted slopes do not exceed 0.1.}
\label{supp:fig:total_time}
\end{figure*}


\end{document}